\documentclass{article} % For LaTeX2e
\usepackage{iclr2024_conference1,times}

\usepackage{amsmath,amsfonts,bm}

\def\eqref#1{equation~\ref{#1}}
\def\1{\bm{1}}

\def\vphi{{\bm{\phi}}}
\def\vxi{{\bm{\xi}}}
\def\vzero{{\bm{0}}}

\def\vmu{{\bm{\mu}}}
\def\vtheta{{\bm{\theta}}}

\def\vf{{\bm{f}}}

\def\vp{{\bm{p}}}
\def\vq{{\bm{q}}}

\def\vs{{\bm{s}}}

\def\vu{{\bm{u}}}
\def\vv{{\bm{v}}}
\def\vw{{\bm{w}}}
\def\vx{{\bm{x}}}
\def\vy{{\bm{y}}}
\def\vz{{\bm{z}}}

\def\mI{{\bm{I}}}

\def\mY{{\bm{Y}}}
\def\mZ{{\bm{Z}}}

\DeclareMathAlphabet{\mathsfit}{\encodingdefault}{\sfdefault}{m}{sl}
\SetMathAlphabet{\mathsfit}{bold}{\encodingdefault}{\sfdefault}{bx}{n}

\newcommand{\pdata}{p_{\rm{data}}}
\usepackage{hyperref}
\usepackage{url}
\usepackage{amsmath}
\usepackage{amsthm}
\usepackage{amssymb}

\usepackage{tikz}

\title{Convergence guarantee for consistency models}

\author{Junlong Lyu   \\
Noah's Ark Lab, HK\\
Huawei Tech.\\
\texttt{lyujunlong@huawei.com} 
\And
 Zhitang Chen  \\
Noah's Ark Lab, HK\\
Huawei Tech.\\
\texttt{chenzhitang2@huawei.com} 
\And
Shoubo Feng \\
Noah's Ark Lab \\
Huawei Tech. \\
\texttt{fengshoubo1@huawei.com} \\
}

\def\dd#1{\mathrm{d}#1}

\newtheorem{thm}{Theorem}
\newtheorem{cor}[thm]{Corollary}
\newtheorem{asmp}{Assumption}
\newtheorem{lem}[thm]{Lemma}

\newtheorem{rmk}{Remark}

\iclrfinalcopy % Uncomment for camera-ready version, but NOT for submission.
\begin{document}

\maketitle

\begin{abstract}
We provide the first convergence guarantees for the Consistency Models (CMs),  a newly emerging type of one-step generative models that can generate comparable samples to those generated by Diffusion Models. Our main result is that, under the basic assumptions on score-matching errors, consistency errors and smoothness of the data distribution, CMs can efficiently sample from any realistic data distribution in one step with small $W_2$ error. Our results (1) hold for $L^2$-accurate score and consistency assumption (rather than $L^\infty$-accurate); (2) do note require strong assumptions on the data distribution such as log-Sobelev inequality; (3) scale polynomially in all parameters; and (4) match the state-of-the-art convergence guarantee for score-based generative models (SGMs). 
We also provide the result that the Multistep Consistency Sampling procedure can further reduce the error comparing to one step sampling, which support the original statement of \cite{Song2023ConsistencyM}. 
Our result further imply a TV error guarantee when take some Langevin-based modifications to the output distributions. 
\end{abstract}

\section{Introduction}

Score-based generative models (SGMs), also known as diffusion models (\cite{SohlDickstein2015DeepUL,Song2019GenerativeMB,Dhariwal2021DiffusionMB,Song2021MaximumLT, Song2020ScoreBasedGM}), are a family of generative models which achieve unprecedented success across multiple fields like image generation (\cite{Dhariwal2021DiffusionMB,Nichol2021GLIDETP,Ramesh2022HierarchicalTI,Saharia2022PhotorealisticTD}), audio synthesis (\cite{Kong2020DiffWaveAV,Chen2020WaveGradEG,Popov2021GradTTSAD}) and video generation (\cite{Ho2022ImagenVH,Ho2022VideoDM}). A key point of diffusion models is the iterative sampling process which gradually reduce noise from random initial vectorswhich provides a flexible trade-off of compute and sample quality, as using extra compute for more iterations usually yields samples of better quality. However, compared to single-step generative models like GANs, VAEs or normalizing flows, the generation process of diffusion models requires 10-2000 times more, causing limited real-time applications.

To overcome this limination, \cite{Song2023ConsistencyM} proposed  Consistency Models (CMs) that can directly map noise to data, which can be seen as an extention of SGMs. CMs support fast one-step generation by design, while still allowing multistep sampling  to trade compute for sampling quality. CMs can be trained either by distilling pre-trained diffusion models, or as stand alone generative models altogether. \cite{Song2023ConsistencyM} demonstrate its superiority through extensive experiments outperforming existing distillation techniques for diffusion models, and when trained in isolation, CMs outperform existing one-step, non-adversarial generative models on standard benchmarks.

Besides the achievements of CMs in saving the the generation costs as well as keeping the sampling quality, it is a pressing question of both practical and theoretical concern to understand the mathematical underpinnings which explain their startling successes. The theoretical guarantee for SGMs has been extensively studied and well established: super-polynomial convergence...  strong assumptions on the data distribution... strong assumptions on the score estimation error... Despite the theoretical successes of SGMs, one would wonder if CMs can inherit the good points from SGMs, as they are inextricably linked in their underlying mathematical process. 

Providing a convergence analysis for CMs and exploring the benefits for CMs comparing to SGMs is a pressing first step towards theoretically understanding why CMs actually work in practice. 

\subsection{Our contributions}
In this work, we take a step towards connecting theory and practice by providing a convergence guarantee for CMs, under minimal assumptions that coincide with our intrinsic. For the underlying SGMs, we make no more assumptions than the state-of-the-art works:

{\bf A1}  The score function of the forward process is $L_s$-Lipschitz. \\
{\bf A2}  The second moment of the data distribution $\pdata$ is bounded.

Note that these two assumptions are standard, no more than what is needed in prior works. The crucial point to these two asumptions is that they do not need log-concavity, a log-Sobelev inequality, or dissipativity, which cover arbitrarily non-log-concave data distributions. Our main result is summarized informally as follows.
\begin{thm}\label{thm1}
    Under Assumptions {\bf A1} and {\bf A2},  and in addition if the consistency model is Lipschitz and the consistency error, score estimation error in $L^2$ are at most $O(\varepsilon)$ with an appropriate choice of step size in training procedure, then, the CM outputs a measure which is $\varepsilon$-close in Wasserstein-2 ($W_2$) distance to $\pdata$ in single step.
\end{thm}
We find Theorem \ref{thm1} surprising, because it shows that CMs can output a distribution arbitrary close to the data distribution in $W_2$ distance with a single step. The error of CMs is just the same order as what SGMs achieved, under the assumption that the consistency error is small enough, which coincide with the incredible success of CMs in many benchmarks. In the fields of neural networks, our result implies that so long as the neural network succeeds at the score and consistency function estimation tasks, the remaining part of the CM algorithm is almost understood, as it admits a strong theoretical justification.

However, learning the score function and consistency function is also difficult in general. Nevertheless, our result still leads the way to further investigations, such as: do consistency function for real-life data have intrinsic structure which can be well-explored by neural networks? If the answer is true, this would then provide an end-to-end guarantee for CMs.

{\bf Better error bound by multistep consistency sampling.} Beyond the one step CM sampling, \cite{Song2023ConsistencyM} suggests a way for multistep sampling to trade compute for sampling quality. However in their original work, they did not make any analysis on the positive effect of multistep sampling comparing to one-step sampling. In our work, we analysis the error bound for each middle state of multistep sampling, showing an asymptotically convergent error bound that greatly smaller than the error bound for one-step sampling. Our analysis reveals the fact that with a suitable choice of time points, the multistep consistency sampling pay a few steps to achieve the near-best performance.

{\bf Bounding the Total Variational error.} As the mathematical foundation of CMs established on the reverse probability flow ODE, they share the same shortcoming comparing to the probability flow SDE: they can not get an error bound in Total Variational (TV) distance or Kullback-Leibler divergence by only controlling the score-matching objective, thus may even fail to estimate the likelihood of very simple data distributions (Fig.1 in \cite{Lu2022MaximumLT}). To solve this potential problem, we offered two  modification process to control the TV error: we can take an OU-type smoothing which takes no more evaluation costs to get a relatively larger TV error bound; we can also apply a Langevin dynamics for correcting purpose with the score model to get a smaller TV error bound, while it needs additional  $O(\varepsilon^{-1} d^{1/2})$ evaluation steps. 

{\bf Bounding the ${\bm W_{\bf 2}}$ error for general bounded data distribution.} The foregoing results are established on the assumption {\bf A1}, which only holds true for $L_s$-smooth data distribution. When only assume the data distribution to be bounded supported, which includes the situation when $q$ is supported on a lower-dimensional submanifold of $\mathbb{R}^d$, we can still guarantee polynomial convergence in the Wasserstein metric by early stopping. As the methodology is  the same as in \cite{Chen2022SamplingIA} and \cite{Lee2022ConvergenceOS}, we do not 
claim the originality, but just include this part for completeness.

\subsection{Prior works}

As far as we known, this is the first work to establish a systematical analysis to the convergence property of CMs. As the CMs and SGMs share the similar mathematical essence in asymptotic situation, our result can be compared to a vast list of literatures on convergence of SGMs. 

{\bf SDE-type SGMs.} The Langevin Monte Carlo (LMC) algorithm (\cite{Rossky1978BrownianDA}) can be seen as the predecessor to the SDE-type SGMs, and literature on convergence of LMCs is extensive, such as \cite{Durmus2015NonasymptoticCA, Cheng2017ConvergenceOL, Cheng2017UnderdampedLM}. However, these works mainly consider the case of exact or stochastic gradients. By the structure of score-matching loss function, only $L^2$-accurate gradient can be guaranteed for SDE-type SGMs. [\cite{Lee2022ConvergenceFS}] are the first to give polynomial convergence guarantee in TV distance under $L^2$-accurate score. However, they rely on the data distribution satisfying smoothness conditions and a log-Sobolev inequality, which essentially limits the guarantees to unimodal distributions.

\cite{Bortoli2022ConvergenceOD} instead only make minimal data assumptions, giving convergence in Wasserstein distance for distributions with bounded support $\mathcal{M}$. In particular, this covers the case of distributions supported on lower-dimensional manifolds, where guarantees in TV distance are unattainable. However, their guarantees have exponential dependence on the diameter of $\mathcal{M}$ or other parameters such as the Lipstchitz constant of score function.

Recently, \cite{Chen2022SamplingIA} and \cite{Lee2022ConvergenceOS} concurrently obtained theoretical guarantees for SGMs under similar general assumptions on the data distribution. They give Wasserstein bounds for any distribution of bounded support (or sufficiently decaying tails), and TV bounds for distributions under minimal smoothness assumptions, that are polynomial in all parameters. This gives theoretical grounding to the success of SGM of data distribution that are often non-smooth and multimodal.

{\bf ODE-type SGMs.} Instead of implementing the time reversed diffusion as an SDE, it is also possible to implement it as an ordinary differential equation (ODE). However, currerent analyses of SGMs cannot provide a polynomial-complexity TV bounds of the probability flow ODE under minimal assumption on data distribution. \cite{Lu2022MaximumLT} first bounded the KL divergence gap (and thus TV error) by higher order gradients of the score function, and thus suggest controlling this bound by minimizing the higher order score-matching objectives, which causes much more difficulties in training the score model.

Instead of changing the training procedure, \cite{Chen2023RestorationDegradationBL}  obtained a discretization analysis for the probability flow ODE in KL divergence, though their bounds have a large dependency on $d$, exponential in the Lipschitz constant of the score integrated over time, which rely on higher order regularities of the log-data density.

To overcome the difficulty on strong data density regularities assumptions, \cite{Chen2023ThePF} only assume the data density to be $L$-smoothness and suggest to interleave steps of the discretized probability flow ODEs with Langevin diffusion correctors using the estimated score, and get a better convergence guarantee than SDEs thanks to the $\mathcal{C}^1$ trajectory for ODEs comparing to the $\mathcal{C}^{\frac{1}{2}-}$ trajectory for SDEs.

\section{Preliminary} \label{pre}
\subsection{Diffusion Models}
Consistency models are heavily relied on the denoising diffusion probabilistic modeling (DDPM). We start with a forward process defined in $\mathbb{R}^d$, which is expressed as a stochastic differential equation 
\begin{equation}
    d{\vx}_t = {\vmu} ( {\vx}_t, t)\dd{t} + \sigma(t) d {\vw}_t
\end{equation}
where $t \in [0,T], T > 0$ is a fixed constant, $\vmu(\cdot,\cdot)$ and $\sigma(\cdot)$ are the drift and diffusion coefficients respectively, and $\{\vw_t\}_{t \in [0,T]}$ denotes the $ d$-dimensional standard Brownian motion. Denote the disbribution of $\vx_t$ as $p_t(\vx)$, therefore $p_0(\vx) = \pdata(\vx)$. A remarkable property is the existence of an ordinary differential equation dubbed the \textit{probability flow ODE}, whose solution trajectories sampled at $t$ are distributed according to $p_t(\vx)$:
\begin{equation} \label{pfode-general}
    \dd{\vx}_t = \left[ \vmu ( \vx_t, t) - \frac{1}{2} \sigma(t)^2 \nabla \log p_t(\vx_t) \right ] d t
\end{equation}
here $\nabla \log p_t(\vx)$ is the \textit{score function} of $p_t(\vx)$.

For clarity, we consider the simplest possible choice, which is the Ornstein-Uhlenbeck (OU) process as in \cite{Chen2023ThePF}, where $\mu(\vx,t) = -\vx$ and $\sigma(t) \equiv \sqrt{2}$,
\begin{equation} \label{pfsde-exact}
    \dd{\vx}_t = -\vx_t \dd t + \sqrt{2} \dd \vw_t, \vx_0 \sim \pdata,
\end{equation}
The corresponding backward ODE is
\begin{equation} \label{pfode-exact}
    \dd{\vx}_t = (-\vx_t - \nabla \log p_t(\vx) )\dd t.
\end{equation}

In this case we have 
\begin{equation}p_t(\vx) = e^{dt}\pdata(e^{t}\vx) \ast \mathcal{N}(\bm{0},(1-e^{-2t})  \mI_d), \label{OUscheduler} \end{equation}
where $\ast$ denotes the convolution operator. We take $\pi(\vx) = \mathcal{N}(\vzero, \mI_d)$, which is a tractable Gaussian distribution close to $p_T(\vx)$. For sampling, we first train a score model $\vs_{\vphi}\approx \nabla \log p_t(\bm x)$ via score matching (\cite{Hyvrinen2005EstimationON,Song2019GenerativeMB,Ho2020DenoisingDP}), then plug into \eqref{pfode-exact} to get the empirical estimation of the PF ODE, which takes the form of
\begin{equation} \label{pfode-empirical}
    \dd{\vx}_t= (-\vx_t -\vs_{\vphi}(\vx_t,t)) \dd{t}.
\end{equation}
We call  \eqref{pfode-empirical} the \textit{empirical PF ODE}. Denote the distribution of $\vx_t$ in \eqref{pfode-empirical} as $q_t(\vx)$. Empirical PF ODE gradually transforms $q_T(\vx) = \pi(\vx)$ into $q_0(\vx)$, which can be view as an approximation of $\pdata(\vx)$.
% , or more precisely, $p_\epsilon(\vx) = \pdata(\vx)\ast\mathcal{N}(\bm{0},\epsilon^2  \mI_d)$.

\subsection{Consistency Models}
For any ordinary differential equation defined on $\mathbb{R}^d$ with vector field $\vv:\mathbb{R}^d \times \mathbb{R}^+ \to \mathbb{R}^d$,
\begin{equation} \nonumber
\dd{\vx}_t = \vv(\vx_t,t)\dd{t}, 
\end{equation} we may define the associate backward mapping $\vf^\vv: \mathbb{R}^d \times \mathbb{R}^+ \to \mathbb{R}^d$ such that
\begin{equation}\label{consistency}
    \vf^\vv(\vx_t, t) = \vx_\delta.
\end{equation}
with an early-stopping time $\delta>0$.
Under mild condition on $\vv$, such a mapping $\vf^\vv$ exists for any $t\in \mathbb{R}^+$, and is smooth relied on $\vx$ and $t$. Note that  \eqref{consistency} is equivalent to the following condition
\begin{equation} \nonumber
    \vf^\vv(\vx_t,t) = \vf^\vv(\vx_s,\tau),~\forall~ 0 \le \tau,t \le T 
\end{equation}
which playing the essential role in constructing consistency loss.

Now let us take $\vv^{\text{ex}}(\vx,t) = -\vx - \nabla \log p_t(\vx)$ and $\vv^{\text{em}}(\vx,t) = -\vx - \vs_{\bm \vphi}(\vx_t,t)$, and denote the corresponding backward mapping function $\vf^{\text{ex}}$, for exact vector field $\vv^{\text{ex}}$,  or $\vf^{\text{em}}$, for empirical vector field $\vv^{\text{em}}$, respectively. As we only have access to the $\vv^{\text{em}}$, we aim to construct a parametric model $\vf_\vtheta$ to approximate $\vf^{\text{em}}$. The authors first implement the boundary condition using the skip connection,
\begin{equation} \nonumber
    \vf_\vtheta(\vx,t)= c_{\text{skip}}(t)\vx + c_{\text{out}}(t)F_\vtheta(\vx,t)
\end{equation}
with differentiable $c_{\text{skip}}(t),c_{\text{out}}(t)$ such that $c_{\text{skip}}(\delta) = 1, c_{\text{out}}(\delta) = 0$, and then define the following Consistency Distillation object:
\begin{equation}\label{distillation_loss}
    \mathcal{L}_{\text{CD}}^N(\vtheta,\vtheta^-;\vphi):= \mathbb{E}[\lambda(t_n) \Vert \vf_\vtheta(\vx_{t_{n+1}}, t_{n+1})- \vf_{\vtheta^-}(\hat\vx_{t_n}^\vphi,t_n)\Vert_2^2],
\end{equation}
where $0  < t_1 = \delta <t_2 \cdots < t_N = T$, $n$ uniformly distributed over $\{1,2,\cdots,N-1\}$, $\vx \sim \pdata$, and $\vx_{t_{n+1}} \sim \mathcal{N}(\vx; t_{n+1}^2 \mI_d)$. Here $\hat\vx_{t_n}^\vphi$ is calculated by 
\begin{equation} \label{ODEsolver}
    \hat\vx_{t_n}^\vphi:= \Phi(\vx_{t_{n+1}}, t_{n+1}, t_n;\vphi)
\end{equation}
where $\Phi(\cdots;\vphi)$ represents the update function of a ODE solver applied to the empirical PF ODE \ref{pfode-empirical}. We may use the exponential integrator (i.e., exactly integrating the linear part),
\begin{equation}\label{exponential integrator}
    \hat \vx^{\vphi}_{t_n} = e^{t_{n+1} - t_n}\vx_{t_{n+1}} + (e^{t_{n+1} - t_n} - 1) \vs_\vphi(\vx_{t_{n+1}}, t_{n+1}).
\end{equation}

For simplicity, we assume $\lambda(t_n) \equiv 1$, and only consider the square of $l_2$ distance to build the loss, \cite{Song2023ConsistencyM} also considered other distance metric such as $l_1$ distance $\Vert \vx - \vy \Vert_1$, and the Learned Perceptual Image Patch Similarity (LPIPS, \cite{Zhang2018TheUE}).

To stabilize the training process, \cite{Song2023ConsistencyM} introduce an additional parameter $\vtheta^-$ and update it by an exponential moving average (EMA) strategy. That is, given a decay rate $0\le \mu < 1$, the author perform the following update after each optimizaiton step:
$\vtheta^- = \text{stopgrad}(\mu\vtheta^- + (1-\mu) \vtheta)$.

Besides the distillation strategy that needs an existing score model $\vs_{\vphi}$, \cite{Song2023ConsistencyM} also introduced a way to train without any pre-trained score models called the Consistency Training (CT) objective. As the CT objective is not closely related to our topic, we place the expression of CT objective under the OU scheduler \ref{OUscheduler} and the exponential integrator \ref{ODEsolver} in Lemma \ref{CTform}.

\cite{Song2023ConsistencyM} gave a asymptotic analysis on the approximation error in their original work. If $\mathcal{L}_{\text{CD}}^N(\vtheta,\vtheta;\vphi) = 0$, we have $\sup_{n,\vx} \Vert \vf_\vtheta(\vx,t_n) - \vf^{\text{em}}(\vx,t_n) \Vert_2 = O((\Delta t)^p)$ when the numerical integrator has local error uniformly bounded by $O((\Delta t)^{p+1})$ with $p \ge 1$.  However, when $\mathcal{L}_{
\text{CD}}^N \neq 0$, we also need a 
quantitative analysis on how far between $\vf_\vtheta$ and  $\vf^{\text{em}}$, and further to analysis the distance between generated distribution and true data distribution.

\section{Main Results}
In this section, we formally state our assumptions and our main results. We denote $\vf_{\vtheta,t} (\vx) = \vf_\vtheta(\vx,t)$ to emphasize the transformation of $\vf_\vtheta$ over $\vx$ at time $t$. We summarized some notations in the Appendix \ref{commom symbols}.
\subsection{Assumptions}
We assume the following mild conditions on the data distribution $\pdata$.
\begin{asmp}\label{normalized data}
    The data distribution has finite second moment, that is, $\mathbb{E}_{\vx_0 \sim \pdata}[\Vert\vx_0\Vert_2^2] = m_2^2 < \infty$.
\end{asmp}

\begin{asmp}\label{Lips cond score}
    The score function $\nabla \log p_t(\vx)$ is Lipschitz on the variable $\vx$ with Lipschitz constant $L_s \ge 1$, $\forall t \in [0,T]$. 
\end{asmp}

This two assumptions are standard and has been used in prior works \cite{Block2020GenerativeMW, Lee2022ConvergenceFS, Lee2022ConvergenceOS,Chen2022SamplingIA}. As \cite{Lee2022ConvergenceOS,Chen2022SamplingIA}, we do not assume Lipschitzness of the score estimate; unlike  \cite{Block2020GenerativeMW, Bortoli2021DiffusionSB}, we do not assume any convexity or dissipativity assumptions on the potential $U = -\log(\pdata)$, and unlike \cite{Lee2022ConvergenceFS} we do not assume $\pdata$ satisfies a log-Sobolev inequality. Thus our assumptions are general enough to cover the highly non-log-concave data distributions. Our assumption could be further weaken to only be compactly supported, as we stated in Section \ref{General data assumption}.

We also assume bounds on the score estimation error and consistency error.
\begin{asmp}\label{score err}
    Assume $\mathbb{E}_{\vx_{t_n} \sim p_{t_n}} [\Vert \vs_{\vphi}(\vx_{t_n},t_n) - \nabla \log p_{t_n}(\vx_{t_n}) \Vert_2^2] \le \varepsilon_{\text{sc}}^2, \forall n \in [\![1, N]\!]$.
\end{asmp}
\begin{asmp}\label{consistency err}
Assume $\mathbb{E}_{\vx_{t_n} \sim p_{t_n}}[\Vert \vf_\vtheta(\vx_{t_{n+1}}, t_{n+1})- \vf_{\vtheta}(\hat\vx_{t_n}^\vphi, t_n)\Vert_2^2]\le \varepsilon_{\text{cm}}^2(t_{n+1} - t_n)^2, \forall n \in [\![1,N-1]\!]$, where $\hat \vx^{\vphi}_{t_n}$ is the exponential integrator defined as in \eqref{exponential integrator}.
\end{asmp}
The score estimation error is the same as in \cite{Lee2022ConvergenceFS,Chen2022SamplingIA}. As discussed in Section \ref{pre}, these two assumption are nature and realistic in light of the derivation of the score matching objective and consistency distillation object.

The following Lipschitz condition for the consistency model is nature and has been used in prior work ( \cite{Song2023ConsistencyM}, Theorem 1).
\begin{asmp}\label{Lips cond consis}
    The consistency model $\vf_{\vtheta}(\vx,t_n)$ is Lipschitz on the variable $\vx$ with Lipschitz constant $L_f > 1,~\forall n \in [\![1,N]\!].$
\end{asmp}

For technique reason, we divide our discretization schedule into two stages: in the first stage, which lasts from $T$ to $h$, we keep the step size equal to $h$; in the second stage, which lasts from $h$ to $\delta$, we take a geometric reducing sequence $2^{-1}h, 2^{-2}h, \cdots$ until $2^{-l}h \le \delta$ for some $l \ge 1$.
\begin{asmp}\label{Timestep schedule}
    Assume the discretization schedule $0 < \delta = t_1 <t_2 < \cdots <t_N = T$, $h_k = t_{k+1} - t_{k}$ to \eqref{pfode-empirical} is divided into two stages:
    \begin{enumerate}
        \item $h_k \equiv  h $ for all $k \in [\![ N_1, N-1]\!],$ and $(N-N_1-1)h < T \le (N-N_1)h$;
        \item $h_k = 2^{-(N_1-k)} h = \frac{h_{k+1}}{2}$ for $k \in [\![1,N_1-1]\!]$, $N_1$ satisfies $h_2 = 2^{-(N_1-2)} h \le 2\delta. $
    \end{enumerate}
    note that in this case $h_1 = T - \sum_{k=2}^{N}h_k - \delta \le h - (1-2^{-(N_1-2)} h)h - \delta  \le \delta$, and $t_{N_1} \le h$
\end{asmp}
\begin{figure}[h]
    \centering
    \begin{tikzpicture} 
    \draw[->] (-1,0) -- (11,0);
    \draw[->] (0,0) --(0,-0.2) node[below = 3.6pt,scale = 0.7]{0};
    \draw[->] (0.125,0) --(0.125,0.2) node[above = 3.6pt,scale = 0.7]{$t_{1}$};
    \draw[->] (0.25,0) --(0.25,-0.2) node[below = 3.6pt,scale = 0.7]{$t_{2}$};
    \draw[->] (0.5,0) --(0.5,0.2) node[above = 3.6pt,scale = 0.7]{$t_{3}$};
    \draw[->] (1,0) --(1,-0.2) node[below = 3.6pt,scale = 0.7]{$t_{4}$};
    \draw (1.5,0) -- (1.5,0) node[below = 10pt,scale = 0.7]{$\cdots$};
    \draw[->] (2,0) --(2,-0.2) node[below = 3.6pt,scale = 0.7]{$t_{N_1}$};
    \draw[->] (4,0) --(4,-0.2) node[below = 3.6pt,scale = 0.7]{$t_{N_1+1}$};
    \draw[->] (6,0) --(6,-0.2) node[below = 3.6pt,scale = 0.7]{$t_{N_1+2}$};
    \draw (8,0) --(8,0) node[below = 10pt,scale = 0.7]{$\cdots$};
    \draw[->] (10,0) --(10,-0.2) node[below = 3.6pt,scale = 0.7]{$t_{N} = T$};
    \end{tikzpicture}
    \caption{Illustration of the discretization schedule in Assumption \ref{Timestep schedule}}
    \label{Timestep schedule fig}
\end{figure}
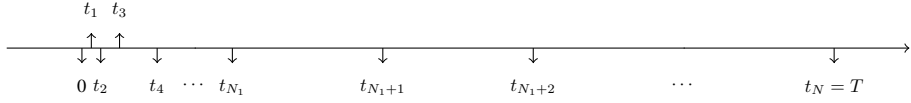

\subsection{$W_2$ error guarantee for one-step Consistency Generating}
In this section we introduce our main results. The first result bounds the Consistency Model estimation error in an expectation mean. 
\begin{thm}[see Section \ref{CMErrorproof} of Appendix] \label{thm-mapping-error}
Under Assumptions \ref{normalized data}-\ref{Lips cond consis}, assume we choose $\Phi$ as the exponential integrator, and assume the timestep schedule satisfies assumption \ref{Timestep schedule},
for $1 \le n \le N-1$:
    \begin{equation} \nonumber
        \left(\mathbb{E}_{\vx_{t_n} \sim p_{t_n} } [\Vert \vf_\vtheta(\vx_{t_n},t_n) - \vf^{\text{ex}}(\vx_{t_n},t_n)\Vert_2^2] \right)^{1/2}\lesssim t_n (\varepsilon_{\text{cm}} + L_f \varepsilon_{\text{sc}} + L_f L_s^{\frac{3}{2}} d^{\frac{1}{2}} h) + t_n^{\frac{1}{2}} L_f L_s d^{\frac{1}{2}} h.
    \end{equation}
\end{thm}
Now we can get our first theorem that analysis the $W_2$ distance after CM mapping. 
\begin{thm}[see Section \ref{Onestep W2 error} of Appendix] \label{thm-W2-error}
Under Assumptions \ref{normalized data}-\ref{Timestep schedule}, let $\mu(\vx)$ be any probability density, $p_t(\vx) = e^{dt}\pdata(e^{t}\vx) \ast \mathcal{N}(\bm{0},(1-e^{-2t})  \mI_d),$ then the following estimation holds,
\begin{equation}\label{eq-W2-error}
    W_2( \vf_{\vtheta,t_n}\sharp \mu, p_{\delta} ) \lesssim L_f W_2(\mu,p_{t_n})+ t_n (\varepsilon_{\text{cm}} + L_f \varepsilon_{\text{sc}} + L_f L_s^{\frac{3}{2}} d^{\frac{1}{2}} h) + t_n^{\frac{1}{2}} L_f L_s d^{\frac{1}{2}} h.
\end{equation}
\end{thm}
As a consequence, we directly get the one-step generation error.
\begin{cor}[see Section \ref{Onestep W2 error} of Appendix] \label{thm-onestep-error}
    Under Assumptions \ref{normalized data}-\ref{Timestep schedule}, when $T > L_s^{-1}$, the one-step generating error is bounded as follows,
    \begin{equation}\label{onestep err to zero}
        W_2(\vf_{\vtheta,T}\sharp \mathcal{N}(\vzero,\mI_d),\pdata) \lesssim (d^{\frac{1}{2}} \vee m_2) L_f e^{-T} + T (\varepsilon_{\text{cm}} + L_f \varepsilon_{\text{sc}} + L_f L_s^{\frac{3}{2}} d^{\frac{1}{2}} h)  + (d^{\frac{1}{2}} \vee m_2)\delta^{\frac{1}{2}}.
    \end{equation}
    In particular, for any $\varepsilon>0$, if we set $\delta \asymp \frac{\varepsilon^2}{d\vee m_2^2}$, $T \ge O(\log(\frac{L_f(\sqrt{d}\vee R)}{\varepsilon})),$ step size $h = O(\frac{\varepsilon}{TL_fL_s^{3/2}d^{1/2}})$, $\varepsilon_\text{cm} = O(\frac{\varepsilon}{T}), \varepsilon_\text{sc} = O(\frac{\varepsilon}{L_fT})$, we can guarantee $ W_2(\vf_{\vtheta,T}\sharp \mathcal{N}(\vzero,\mI_d),\pdata) \lesssim \varepsilon$.
\end{cor}

We remark that our discretization complexity $N = O(\frac{T}{h}\log(\frac{1}{h})) = O(\frac{L_f L_s^{3/2}d^{1/2}}{\varepsilon})$ matches state-of-the-art complexity for ODE-type SGMs \cite{Chen2023ThePF,Chen2023RestorationDegradationBL}. This provides some evidence that our descretization bounds are of the correct order.

Note that the error bound in \ref{onestep err to zero} relied on the final time $T$, consistency error $\varepsilon_{\text{cm}}$, score error $\varepsilon_{\text{sc}}$ and step size $h$. Actually we can refine the error and reduce the linear dependency on $T$ to log dependency by Multistep Consistency Sampling that will be introduced in the next section.

\subsection{Multistep Consistency Sampling can reduce the $W_2$ error}
Now let us analysis the effect of Multistep Consistency Sampling introduced in original CM work, Algorithm 1, \cite{Song2023ConsistencyM}, which has been introduced to improve the sample quality by alternating denoising and noise injection steps. Given a sequence of time points $T = t_{n_1} \ge t_{n_2}\ge \cdots$, adapted to the OU noise scheduler \ref{pfsde-exact} , the generating procedure can be written as
\begin{align}
    \vz_1 &:= \vf_\vtheta(\vxi_1, T), \nonumber\\
    \vu_k &:= e^{-(t_{n_k} - \delta)} \vz_{k-1} + \sqrt{1-e^{-2(t_{n_k} - \delta)}} \vxi_k, \nonumber\\
    \vz_k &:= \vf_{\vtheta}(\vu_k, t_{n_k}),
\end{align}
with $\vxi_i$ i.i.d. $N(\vzero,\mI_d)$ distributed, 
and  $q_k:= \text{law}(\vz_k)$ satisfies the following relationship:
\begin{align}
    q_1 &= \vf_{\vtheta,{T}} \sharp \mathcal{N}(\vzero,  \mI_d), \nonumber\\
    \mu_k &:= \left(e^{d(t_{n_k} - \delta)}q_{k-1}(e^{t_{n_k} - \delta}\vx)\right) \ast \mathcal{N}(\vzero, (1-e^{-2(t_{n_k} - \delta)}) \mI_d). \nonumber\\
    q_k &= \vf_{\vtheta,t_{n_k}} \sharp \mu_k, \label{multistep output}
\end{align}
We thus have the following upper bound of the $W_2$ distance between $q_k$ and $p_\delta$. 

\begin{cor}[see Section \ref{Multistep proof} of Appendix] \label{multisteperr}
Under Assumptions \ref{normalized data}-\ref{Timestep schedule},  when $T > L_s^{-1}$, the $W_2$ distance between $q_k$ and $p_\delta$ can be controlled by $q_{k-1}$ and $p_\delta$ as follows,
\begin{equation}\label{W2bound-reccur}
    W_2(q_k, p_\delta) \lesssim  L_fe^{-t_{n_k}} W_2(q_{k-1},p_\delta) +  t_{n_k} (\varepsilon_{\text{cm}} + L_f \varepsilon_{\text{sc}} + L_f L_s^{\frac{3}{2}} d^{\frac{1}{2}} h).
\end{equation}
\end{cor}

Note that in \eqref{W2bound-reccur}, we have an exponentially small multiplier $e^{-t_{n_k}}$ gradually reduce the error introduced from the previous steps, and  a $t_{n_k}$-linear term representing the error introduced from the current step. By choosing a suitable time schedule $\{t_{n_k}\}_{k\ge 1}$, we can get a finer bound in $W_2$.

\begin{cor}[see Section \ref{Multistep proof} of Appendix]\label{MultistepW2Err}
Under Assumptions \ref{normalized data}-\ref{Timestep schedule}, there exists $T \ge t_{\hat n} \ge \max(\log (2L_f) + \delta, L_s^{-1})  $,  $\hat n \in [\![1,N]\!]$,  such that when taking $n_k \equiv \hat n$ for all $k$,
% $$ W_2(q_k, p_\delta)\lesssim  (\log( L_f) + 2^{-k}T) (\varepsilon_{\text{cm}} + L_f\varepsilon_{\text{sc}} + L_f L_s^{\frac{3}{2}}d^{\frac{1}{2}}h) + 2^{-k}(d^{\frac{1}{2}}\vee m_2)L_fe^{-T},$$
% and
$$ W_2(q_k, \pdata )\lesssim  (\log( L_f) + 2^{-k}T) (\varepsilon_{\text{cm}} + L_f\varepsilon_{\text{sc}} + L_f L_s^{\frac{3}{2}}d^{\frac{1}{2}}h) + 2^{-k}(d^{\frac{1}{2}}\vee m_2)L_fe^{-T}+  (d^{\frac{1}{2}}\vee m_2)\delta^{\frac{1}{2}}.$$
Thus, for any $\varepsilon>0$, if we set $\delta \asymp \frac{\varepsilon^2}{d\vee m_2^2}$,  $k = O(\log(T \vee (\frac{(d\vee m_2^2)L_f}{\varepsilon})))$,  $h = O(\frac{\varepsilon}{\log(L_f)L_fL_s^{3/2}d^{1/2}})$, $\varepsilon_\text{cm} = O(\frac{\varepsilon}{\log(L_f)}), \varepsilon_\text{sc} = O(\frac{\varepsilon}{L_f\log(L_f)})$, we can guarantee $ W_2(q_k,\pdata) \lesssim \varepsilon$.
\end{cor}
\begin{rmk}
Comparing the result between multistep sampling error \ref{MultistepW2Err} and one step sampling error \ref{thm-onestep-error}, the main improvement from multistep sampling is getting rid of the dependency from $T$. In one step sampling, one should take the step size smaller, and train the consistency model and score model better. Besides, multistep sampling \ref{MultistepW2Err} only requires $T 
\ge \max(\log (2L_f) + \delta, L_s^{-1})$, while one step sampling \ref{thm-onestep-error} requires $T \ge O(\log(\frac{L_f(\sqrt{d}\vee R)}{\varepsilon}))$, which is an added benefit that multistep sampling requires lower training complexity comparing to one step sampling.
\end{rmk}

\subsection{Bounding the TV error}
In the sections before, we have showed that the generated distribution of Consistency Models are close to the true data distribution in the metric of Wasserstein-2 distance. When we turn to the Total Variational (TV) distance, however, the error bound is deficient as the situation for the probability flow ODEs, in contrast to the situation for the probability flow SDEs. Here we introduce two 
operations that can further bound the TV error.
\subsubsection{Bounding the TV error by forward OU process} \label{OUsmoothing}
Let the forward OU process be 
\begin{equation} \label{forwardOU}
    \dd{\vx}_t = -\vx_t \dd t + \sqrt{2} \dd \vw_t,
\end{equation}
and denote the Markov kernel $P_{\text{OU}}^s$ to be defined by the equation \ref{forwardOU}, that is, if $\vx_t \sim p$ for some $p$ be a distribution over $\mathbb{R}^d$, $\vx_{t+s} \sim p P_{\text{OU}}^s$. Let $q$ be the output of our Consistency Models, either the one step consistency sampling result, or the k-th multistep consistency sampling result. To control the TV error, we smooth the generated sample by the forward OU process with a small time that is the same as the early stopping time $\delta$, and then we can get the TV distence between $q P_{\text{OU}}^\delta$ and $\pdata$:

\begin{cor}
    [see Section \ref{OUproof} of Appendix] \label{end-to-end TV}
    Under Assumptions \ref{normalized data}-\ref{Timestep schedule}, suppose $q$ is: (1) the one step consistency sampling result, $q = q_1 = \vf_{\vtheta,T}\sharp \mathcal{N}(\vzero,\mI_d)$;  (2) the k-th multistep consistency sampling result, $q = q_k$ defined as in \eqref{multistep output} with multistep schedule as in Corollary \ref{MultistepW2Err}.
    Choose the early stopping time $\delta \asymp \frac{\varepsilon^2}{L_s^2(d\vee m_2^2)}$ for some $\varepsilon > 0$, then if $T \ge \max(\log (2L_f) + \delta, L_s^{-1}), $
    \begin{align}
        \text{TV}&(q_1 P_{\text{OU}}^\delta, \pdata) \nonumber\\
        &\lesssim \frac{L_s L_f (d \vee m_2^2)}{\varepsilon}e^{-T} + \frac{L_s  (d^{\frac{1}{2}} \vee m_2)}{\varepsilon} T (\varepsilon_{\text{cm}} + L_f \varepsilon_{\text{sc}} + L_f L_s^{\frac{3}{2}} d^{\frac{1}{2}} h) + \varepsilon, \nonumber\\
        \text{TV}&(q_k P_{\text{OU}}^\delta, \pdata) \nonumber\\
        &\lesssim \frac{L_s  (d^{\frac{1}{2}} \vee m_2)}{\varepsilon} [(\log L_f + \frac{T}{2^k})   (\varepsilon_{\text{cm}} + L_f \varepsilon_{\text{sc}} + L_f L_s^{\frac{3}{2}} d^{\frac{1}{2}} h)+ \frac{(d^{\frac{1}{2}}\vee m_2)L_f}{2^k e^T}]+ \varepsilon.\nonumber
    \end{align}
    In particular, 
    \begin{enumerate}
        \item If we set $T = O(\log (\frac{(d\vee m_2^2)L_fL_s}{\varepsilon^4}))$, $h = O(\frac{\varepsilon^2}{L_fL_s^{5/2}(d\vee m_2^2)T})$, and if $\varepsilon_{sc} \le O(\frac{\varepsilon^2}{TL_fL_s(d^{1/2} \vee m_2)})$, $\varepsilon_{cm} \le O(\frac{\varepsilon^2}{TL_s(d^{1/2} \vee m_2)})$, then we can guarantee TV error $O(\varepsilon)$ with one step generation and one additional smoothing (with no NN evaluation);
        \item  If we set $k = O(\log(T \vee (\frac{(d\vee m_2^2)L_fL_s}{\varepsilon})))$, $h = O(\frac{\varepsilon^2}{\log(L_f)L_fL_s^{5/2}(d\vee m_2^2)})$ and if $\varepsilon_{sc} \le O(\frac{\varepsilon^2}{\log(L_f)L_fL_s(d^{1/2} \vee m_2)}),$ $\varepsilon_{cm} \le O(\frac{\varepsilon^2}{\log(L_f)L_s(d^{1/2} \vee m_2)})$, then we can guarantee TV error $O(\varepsilon)$ with $k$ steps generations and one additional smoothing (with no NN evaluation).
    \end{enumerate}
\end{cor}

\subsubsection{Bounding the TV error by Underdamped Langevin Corrector}
We may adopt the idea from \cite{Chen2023ThePF}, who introduce the Langevin-correcting procedure into the probability flow ODEs to get a TV error guarantee. 

The Langevin dynamics for correcting purpose is defined as follows: let $p$ be a distribution over $\mathbb{R}^d$, and write $U$ as a shorthand for the potential $-\log p$.

Given a friction parameter $\gamma > 0$,  consider the following discretized process with step size $\tau$, where $-\nabla U$ is replaced by a score estimate $\vs$. Let $(\hat z_t, \hat v_t)_{t\ge 0}$ over $\mathbb{R}^d \otimes \mathbb{R}^d$ be given by
\begin{align} \label{ULMC}
    \dd \hat \vz_t &= \hat\vv_t \dd t, \nonumber\\
    \dd \hat \vv_t &= (\vs(\hat \vz_{\lfloor t/\tau\rfloor \tau}) - \gamma \hat \vv_t) \dd t + \sqrt{2\gamma} \dd \vw_t.
\end{align}

Denote the Markov kernel $\hat P_{\text{ULMC}}$ to be defined by the equation \ref{ULMC}, that is, if $(\hat \vz_{k\tau}, \hat \vv_{k\tau}) \sim \vmu$ for some $\vmu$ be a distribution over $\mathbb{R}^{d\times d}$, $(\hat \vz_{(k+1)\tau}, \hat \vv_{(k+1)\tau}) \sim \vmu \hat P_{\text{ULMC}}$. We denote the $k$-th composition $\hat P_{\text{ULMC}}^k = \hat P_{\text{ULMC}} \circ \hat P_{\text{ULMC}}^{k-1} $, and $\vq = q \otimes \mathcal{N}(0,\mI_d), \vp = p \otimes \mathcal{N}(0,\mI_d).$ In what follows, we abuse the notation as follows. Given a distribution $q$ on $\mathbb{R}^d$, we write $q\hat P_{\text{ULMC}}$ to denote the projection onto the $\vz-$coordinates of $\vq \hat P_{\text{ULMC}}$. It's obvious that 
\begin{equation}\nonumber
    \text{TV}(q\hat P_{\text{ULMC}}^N,p) \le \text{TV}(\vq\hat P_{\text{ULMC}}^N,\vp).
\end{equation}

 Now let's take $p = \pdata$ as the data distribution, and $q = q_k$ for some $k \ge 1$ as the output distribution of $k$-th multistep consistency sampling defined as in \eqref{multistep output}. We can use the score model $\vs_\vphi(\vx, \delta) \approx \nabla \log p_{\delta}(\vx)$ to do corrector steps. The end-to-end error now can be written as follows:

\begin{cor} [see Section \ref{ULMCproof} of Appendix] \label{TV-Langevin-corrector}
    Under Assumptions \ref{normalized data}-\ref{Timestep schedule}, suppose $q$ is: (1) the one step consistency sampling result, $q = q_1 = \vf_{\vtheta,T}\sharp \mathcal{N}(\vzero,\mI_d)$;  (2) the k-th multistep consistency sampling result, $q = q_k$ defined as in \eqref{multistep output} with multistep schedule as in Corollary \ref{MultistepW2Err}.   Choose $\gamma \asymp L_s$, and $\delta \asymp \frac{\varepsilon^2}{L_s^2(d\vee m_2^2)}$ for some $\varepsilon > 0$, then if $T \ge \max(\log (2L_f) + \delta, L_s^{-1}),N\tau \asymp \frac{1}{\sqrt{L_s}}$,
    \begin{align}
        \text{TV}&(q_1\hat P_{\text{ULMC}}^N,\pdata)\nonumber\\
        &\lesssim (d^{\frac{1}{2}}\vee m_2)L_fL_s^{\frac{1}{2}} e^{-T} + TL_s^{\frac{1}{2}} ( \varepsilon_{\text{cm}} + L_f\varepsilon_{\text{sc}} + L_f L_s^{\frac{3}{2}} d^{\frac{1}{2}} h ) + L_s^{-\frac{1}{2}} \varepsilon_{\text{sc}} + L_s^{\frac{1}{2}}d^{\frac{1}{2}}\tau + \varepsilon, \nonumber\\
        \text{TV}&(q_k\hat P_{\text{ULMC}}^N,\pdata) \nonumber\\
        &\lesssim (\log( L_f) + \frac{T}{2^k}) L_s^{\frac{1}{2}}(\varepsilon_{\text{cm}} + L_f\varepsilon_{\text{sc}} + L_f L_s^{\frac{3}{2}}d^{\frac{1}{2}}h) + \frac{(d^{\frac{1}{2}}\vee m_2)L_s^{\frac{1}{2}}L_f}{2^{k}e^T} + L_s^{-\frac{1}{2}} \varepsilon_{\text{sc}} + L_s^{\frac{1}{2}}d^{\frac{1}{2}}\tau+ \varepsilon. \nonumber
    \end{align}
    In particular, 
    \begin{enumerate}
        \item if we set $T = O(\log (\frac{(d\vee m_2^2)L_f^2L_s}{\varepsilon^2}))$, $h = O(\frac{\varepsilon}{L_fL_s^2d^{1/2}T})$, $\tau = O(\frac{\varepsilon}{L_s^{1/2}d^{1/2}})$, and if $\varepsilon_{sc} \le O(\frac{\varepsilon}{TL_fL_s^{1/2}})$, $\varepsilon_{cm} \le O(\frac{\varepsilon}{TL_s^{1/2}})$, then we can obtain TV error $O(\varepsilon)$ with one step consistency generation and $O(\frac{\sqrt{d}}{\varepsilon})$ steps correctings;
        \item  if we set $k = O(\log(T \vee (\frac{(d\vee m_2^2)L_fL_s}{\varepsilon})))$, $h = O(\frac{\varepsilon}{\log(L_f)L_fL_s^2d^{1/2}})$, $\tau = O(\frac{\varepsilon}{L_s^{1/2}d^{1/2}})$ and if $\varepsilon_{sc} \le O(\frac{\varepsilon}{\log(L_f)L_fL_s^{1/2}}),$ $\varepsilon_{cm} \le O(\frac{\varepsilon}{\log(L_f)L_s^{1/2}})$, then we can obtain TV error $O(\varepsilon)$ with $k$ steps consistency generation and  $O(\frac{\sqrt{d}}{\varepsilon})$ steps correctings.
    \end{enumerate}
\end{cor}

\begin{rmk}
In case of missing the score model $\vs_{\vphi}$ but only remain the consistency model $\vf_{\vtheta}$, such as training the consistency model by CT objective \ref{CTobject} without a pretrained score model, we can also recover a score model $\hat \vs_\vtheta$ that approx $\nabla \log p_{t_2}(\vx)$: We have that $\hat\vs_\vtheta(\vx) := \frac{\vf_\vtheta(\vx, t_2) - e^{h_1}\vx}{e^{h_1} - 1}$ is a score model approximating $\nabla \log p_{t_2}(\vx)$ with $L^2$ error $\sqrt{\varepsilon^2_{\text{cm}} + \varepsilon^2_{\text{sc}}}$, which is proved in Lemma \ref{creatingscore}. 
 Hence, we  may run similar procedure as in Theorem \ref{TV-Langevin-corrector}, where $q$ should firstly be transformed with forward OU process \eqref{pfsde-exact} with a small time $h_1 = t_2 - \delta < \delta$, then apply the Underdamped Langevin Corrector Operator with $\hat\vs_\vtheta(\vx)$. 
\end{rmk}

\subsection{$W_2$ convergence guarantee for arbitrary data distributions with bounded support} \label{General data assumption}

In this section, we give a sufficient condition to satisfy the Assumption \ref{Lips cond score}: in fact for any compactly supported distribution $\pdata, \text{supp }\pdata \subseteq B(\vzero,R)$, for any $t_0$, we can get a positive $L_s(t_0)$, such that  Assumption \ref{Lips cond score} is satisfied for any $t > t_0$ with Lipschitz constant $L_s(t_0)$. This include a wide range of situations even when $p$ do not have smooth density w.r.t. Lebesgue measure such as when $p$ supported on a lower-dimensional submanifold of $\mathbb{R}^d$, which recently investigated in \cite{Bortoli2022ConvergenceOD, Lee2022ConvergenceOS, Chen2022SamplingIA}.

Namely, based on the folloing lemma, we can conduct regularity properties for the score functions.
\begin{lem}[see Section \ref{bound support proof} of Appendix]\label{Hessian Bound Lemma}
    Suppose that $\text{supp } \pdata \subseteq B(\vzero, R)$ where $R \ge 1$, and let $p_t$ denote the law of the OU process at time t, started at $p$: that is, $p_t(\vx) = e^{dt}\pdata(e^{t}\vx) \ast \mathcal{N}(\bm{0},(1-e^{-2t})  \mI_d)$. Then the Hessian of the score function satisfies:
    \begin{equation}\nonumber
        \Vert\nabla^2 \log p_t(\vx) \Vert_{\text{op}} \le \frac{e^{-2t}R^2}{(1-e^{-2t})^2}  + \frac{1}{1-e^{-2t}}.
    \end{equation}
\end{lem}
Note that in our proof of Corollary \ref{thm-onestep-error} and \ref{MultistepW2Err}, we only use the Assumption \ref{Lips cond score} over $t \in [\delta, T]$. Combining Lemma \ref{Hessian Bound Lemma}, we immediately get the following corollary.

\begin{cor}[see Section \ref{bound support proof} of Appendix] \label{Bounded pdata err}
Under Assumptions  \ref{score err}-\ref{Timestep schedule}, suppose that $\text{supp } \pdata \subseteq B(\vzero,R)$ where $R \ge 1$. Let $\delta \asymp \frac{\varepsilon^2}{R^2\vee d}$, then (1) the one-step generating error satisfies $W_2(q_1, \pdata) \lesssim \varepsilon$, provided that $T = O(\log (\frac{L_f (\sqrt{d} \vee R)}{\varepsilon}))$, $h =O(\frac{\varepsilon^7}{d^{1/2}R^3(R^6\vee d^3)L_f T})$,  $\varepsilon_{\text{cm}} = O(\frac{\varepsilon}{T}), \varepsilon_{\text{sc}} = O(\frac{\varepsilon}{L_f T})$ ; (2) the multi-step generating error satisfies $W_2(q_k, \pdata) \lesssim \varepsilon$, provided that  $k = O(\log(T \vee (\frac{(d\vee m_2^2)L_fL_s}{\varepsilon})))$, $T = O(\max(\log (2L_f) + \delta, L_s^{-1}))$,  $h =O(\frac{\varepsilon^7}{d^{1/2}R^3(R^6\vee d^3)L_f \log(L_f)})$,  $\varepsilon_{\text{cm}} = O(\frac{\varepsilon}{\log(L_f)}), \varepsilon_{\text{sc}} = O(\frac{\varepsilon}{L_f \log(L_f)})$.
\end{cor}

\section{Conclusion}
In this work, we provided a first convergence guarantee for CMs which holds true under realistic assumptions ($L^2$-accurate score and consistency function surrogates; arbitrarily data distributions with smooth densities respect to Lebesgue measure, or bounded distributions) and which scale at most polynomially in all relavent parameters. Our results take a step towards explaining the success of CMs. We also provide theoretical evidence that multistep CM sampling technique can further reduce the error comparing to one step CM sampling .

There are mainly three shortcomings in our work. Firstly, our proofs relied on the assuming the lipschitz condition on the surrogate model $\vf_\vtheta$, which is unrealistic. We will keep in improving our results by replacing this condition on the exact consistency function $\vf^{\text{ex}}$ in our future work. Secondly, to bound the TV error, we introduced an additional smoothing procedure after the original CM sampling steps. We are seeking better ways to remove this procedure. Lastly, we did not address the question of when the score and consistency function can be learned well. We believe that the resolution of this problem would shed considerable light on CMs.

\subsubsection*{Author Contributions}
If you'd like to, you may include  a section for author contributions as is done
in many journals. This is optional and at the discretion of the authors.

\subsubsection*{Acknowledgments}
Use unnumbered third level headings for the acknowledgments. All
acknowledgments, including those to funding agencies, go at the end of the paper.

\bibliography{iclr2024_conference}
\bibliographystyle{iclr2024_conference}

\appendix
\section{Common Symbols} \label{commom symbols}
\begin{itemize}
    \item $\ast$: the convolution operator defined for two functions $f,g \in L^2(\mathbb{R}^d)$.\\ $f\ast g(\vx) := \int_{\mathbb{R}^d} f(\vx - \vu) g(\vu) d\vu$.
    \item $\lesssim$: less or similar to. If $a \lesssim b$, it means $a \le Cb$ for some constant $C$.
    \item $\sharp$ the push-forward operator associated with a measurable map $f: \mathcal{M} \to \mathcal{N}$. For any measure over $\mathcal{M}$, we may define the push-forward measure $f\sharp \mu$ over $\mathcal{N}$ by: $f\sharp\mu(A) = \mu(f^{-1} (A))$, for any $A$ be measurable set in $\mathcal{N}$.
    \item $\vee$: take the larger one. $a\vee b = \max(a,b)$; \\$\wedge$: take the smaller one. $a\wedge b = \min(a,b)$.
    \item $\asymp$: asymptotic to. If $a_n \asymp b_n$, it means $\lim_{n\to \infty} {a_n}/{b_n} = C$ for some constant $C$.
    \item $[\![ a, b ]\!]:= [a,b]\cap \mathbb{Z}$
\end{itemize}

\section{Proofs }
\subsection{Proofs for Theorem 2}\label{CMErrorproof}
Before we proof our main theorem, we introduce a score perturbation lemma which comes from Lemma 1  in \cite{Chen2023ThePF}.
\begin{lem}[Lemma 1 in \cite{Chen2023ThePF}]\label{scorepurturb}(Score perturbation for $\vx_t$). Suppose $p_t(\vx) = e^{dt}\pdata(e^{t}\vx) \ast \mathcal{N}(\bm{0},(1-e^{-2t})  \mI_d)$ started at $p_0$, and $\vx_0 \sim p_0$, $d \vx_t = -\vx_t - \nabla \log p_t(\vx)$. Suppose that $\Vert \nabla^2 \log p_t(\vx) \Vert_{op} \le L$ for all $\vx \in \mathbb{R}^d$ and $t \in [0,T]$, where $L \ge 1$. Then,
\begin{equation}\nonumber
    \mathbb{E}[\Vert \frac{\partial}{\partial_t} \nabla \log p_t(\vx_t)\Vert_2^2] \lesssim L^2d(L + \frac{1}{t})
\end{equation}
\end{lem}

\begin{proof}[\bf Proof of Theorem \ref{thm-mapping-error}]
We divide the left-hand side by Cauchy-Schwarz inequality as follows: let $\vx_t$ be the solution to the probability flow ODE \ref{pfode-exact}. Notice that $\vf^{\text{ex}}(\vx_{t_n},t_n) = \vx_{\delta} = \vf_{\vtheta}(\vx_{t_1},t_1) $
\begin{align}
    &\left(\mathbb{E}_{\vx_{t_n} \sim p_{t_n} } [\Vert \vf_\vtheta(\vx_{t_n},t_n) - \vf^{\text{ex}}(\vx_{t_n},t_n)\Vert_2^2]\right)^{1/2}  \nonumber\\
    =& \left(\mathbb{E}_{\vx_{t_n} \sim p_{t_n} } \left[\left\Vert \sum_{k = 1}^{n-1} (\vf_\vtheta(\vx_{t_{k+1}}, t_{k+1}) -\vf_\vtheta( \vx_{t_k},t_k)) \right\Vert_2^2\right] \right)^{1/2}\nonumber\\
    \le& \sum_{k = 1}^{n-1} \left(\mathbb{E}_{\vx_{t_n} \sim p_{t_n} } \left[\left\Vert   \vf_\vtheta(\vx_{t_{k+1}}, t_{k+1}) -\vf_\vtheta( \vx_{t_k},t_k) \right\Vert_2^2\right]\right)^{1/2}  \nonumber\\
    =& \sum_{k = 1}^{n-1} \left(\mathbb{E}_{\vx_{t_n} \sim p_{t_n} }\left[\left\Vert \vf_\vtheta(\vx_{t_{k+1}}, t_{k+1}) -\vf_\vtheta( \hat\vx_{t_k}^\vphi,t_k) + \vf_\vtheta( \hat\vx_{t_k}^\vphi,t_k) - \vf_\vtheta( \vx_{t_k},t_k) \right \Vert_2^2\right] \right)^{1/2} \nonumber\\
    \le& \sum_{k = 1}^{n-1} \left(\mathbb{E}_{\vx_{t_n} \sim p_{t_n} }\left[\left\Vert \vf_\vtheta(\vx_{t_{k+1}}, t_{k+1}) -\vf_\vtheta( \hat\vx_{t_k}^\vphi,t_k) \right\Vert_2^2\right]\right)^{1/2} \nonumber\\
    & \quad +  \sum_{k = 1}^{n-1} \left(\mathbb{E}_{\vx_{t_n} \sim p_{t_n} }\left[\left\Vert   \vf_\vtheta( \hat\vx_{t_k}^\vphi,t_k) - \vf_\vtheta( \vx_{t_k},t_k)    \right\Vert_2^2\right]\right)^{1/2} \nonumber\\
    :=& E_1 + E_2 \label{ftotalerror}
\end{align}
We can bound $E_1$ by assumption \ref{consistency err}: note that
\begin{align}
    E_1 = &\sum_{k = 1}^{n-1} \left(\mathbb{E}_{\vx_{t_n} \sim p_{t_n} }\left[\left\Vert \vf_\vtheta(\vx_{t_{k+1}}, t_{k+1}) -\vf_\vtheta( \hat\vx_{t_k}^\vphi,t_k) \right\Vert_2^2\right]\right)^{1/2}  \nonumber\\
    = &\sum_{k = 1}^{n-1} \left(\mathbb{E}_{\vx_{t_{k+1}} \sim p_{t_{k+1}} }\left[\left\Vert \vf_\vtheta(\vx_{t_{k+1}}, t_{k+1}) -\vf_\vtheta( \hat\vx_{t_k}^\vphi,t_k) \right\Vert_2^2\right]\right)^{1/2}   \nonumber\\
    \le& \varepsilon_{\text{cm}} \sum_{k=1}^{n-1} h_k = \varepsilon_{\text{cm}} (t_n - t_1), \label{E1}
\end{align}
where the last equality we use the fact that when $\vx_t$ satisfies \eqref{pfode-exact} and $\vx_{t_n} \sim p_{t_n}, \vx_{t_k} \sim p_{t_k}$ for all $k \le N$.

Now we turn to bounding the second term. We notice that by Lipschitz assumption \ref{Lips cond consis}, 
\begin{align}
    E_2 =& \left(\sum_{k = 1}^{n-1} \mathbb{E}_{\vx_{t_n} \sim p_{t_n} }\left[\left\Vert   \vf_\vtheta( \hat\vx_{t_k}^\vphi,t_k) - \vf_\vtheta( \vx_{t_k},t_k)    \right\Vert_2^2 \right]\right)^{1/2} \nonumber\\
    \le& \sum_{k = 1}^{n-1} L_f \left(\mathbb{E}_{\vx_{t_n} \sim p_{t_n} } \left[\Vert \hat\vx_{t_k}^\vphi -\vx_{t_k}  \Vert_2^2\right]\right)^{1/2} \nonumber\\
    =& \sum_{k = 1}^{n-1} L_f \left(\mathbb{E}_{\vx_{t_{k+1}} \sim p_{t_{k+1}} } \left[\Vert \hat\vx_{t_k}^\vphi -\vx_{t_k}  \Vert_2^2\right]\right)^{1/2} \label{E2}
\end{align}

Now let us bound the term $\Vert \hat\vx_{t_k}^\vphi -\vx_{t_k}  \Vert_2^2$. Note that $\hat\vx_{t_k}^\vphi$ is the exponential integrator solution to the ODE \ref{pfode-empirical}, we have
\begin{align}
    \dd \vx_t &= -(\vx_t + \nabla \log p_t (\vx_t)) \dd t,  \nonumber\\
    \dd \hat \vx_t^\vphi &= -(\hat \vx_t^\vphi + \vs( \hat\vx^{\vphi}_{t_{k+1}},t_{k+1}))\dd t.
\end{align}
for $t_k \le t \le t_{k+1}$ with $\hat \vx_{t_{k+1}}^\vphi = \vx_{t_{k+1}}$. Denote $ h_k = t_{k+1} - t_k$, then,
\begin{align}
    \frac{\partial}{\partial_t} \Vert \hat \vx_t^\vphi- \vx_t \Vert_2^2 &= 2 \langle \hat \vx_t^\vphi- \vx_t, \frac{\partial}{\partial_t}(\hat \vx_t^\vphi- \vx_t)\rangle \nonumber\\
    &= 2\left( \Vert \hat \vx_t^\vphi- \vx_t \Vert_2^2 + \langle \hat \vx_t^\vphi- \vx_t, \vs(\vx_{t_{k+1}},t_{k+1}) - \nabla \log p_t(\vx_t)\rangle \right) \nonumber\\
    &\le (2 + \frac{1}{h_k}) \Vert \hat \vx_t^\vphi- \vx_t \Vert_2^2 + h_k \Vert \vs(\vx_{t_{k+1}},t_{k+1}) - \nabla \log p_t(\vx_t) \Vert_2^2.
\end{align}
By Gr\"onwall's inequality,
\begin{align}
    \mathbb{E}_{\vx_{t_{k+1}} \sim p_{t_{k+1}} } &\left[\Vert \hat\vx_{t_k}^\vphi -\vx_{t_k}  \Vert_2^2\right] \nonumber\\
    &\le \exp((2 + \frac{1}{h_k})h_k) \int_{t_k}^{t_{k+1}} h_k \mathbb{E}_{\vx_{t_{k+1}} \sim p_{t_{k+1}}} \left[\Vert \vs(\vx_{t_{k+1}},t_{k+1}) - \nabla \log p_t(\vx_t) \Vert_2^2\right] \dd t \nonumber\\
    &\lesssim h_k \int_{t_k}^{t_{k+1}} \mathbb{E}_{\vx_{t_{k+1}} \sim p_{t_{k+1}}} \left[\Vert \vs(\vx_{t_{k+1}},t_{k+1}) - \nabla \log p_t(\vx_t) \Vert_2^2\right] \dd t.
\end{align}
We split up the error term as
\begin{align}
    \Vert \vs(\vx_{t_{k+1}},t_{k+1})& - \nabla \log p_t(\vx_t) \Vert_2^2 \nonumber\\
    &\lesssim \Vert \vs(\vx_{t_{k+1}},t_{k+1}) - \nabla \log p_{t_{k+1}}(\vx_{t_{k+1}}) \Vert_2^2 + \Vert \nabla \log p_{t_{k+1}}(\vx_{t_{k+1}}) - \nabla \log p_t(\vx_t) \Vert_2^2.
\end{align}
By assumption \ref{score err}, the first term is bounded in expectation by $\varepsilon_{\text{sc}}^2$. By Lemma \ref{scorepurturb} and $t_k \le t \le t_{k+1}$, the second term is bounded by 
\begin{align}\nonumber
    \mathbb{E}_{\vx_{t_{k+1}} \sim p_{t_{k+1}} }\left[\Vert \nabla \log p_{t_{k+1}}(\vx_{t_{k+1}}) - \nabla \log p_t(\vx_t) \Vert_2^2\right]  &= \mathbb{E}_{\vx_{t_{k+1}} \sim p_{t_{k+1}} }\left[\left\Vert \int_t^{t_{k+1}} \frac{\partial}{\partial_u} \nabla \log p_u(\vx_u) \dd u \right\Vert_2^2\right]\\
    &\le (t_{k+1} - t) \int_t^{t_{k+1}} \mathbb{E}[\Vert\frac{\partial}{\partial_u} \nabla \log p_u(\vx_u)\Vert_2^2] \dd u\\
    &\le h_k \int_t^{t_{k+1}} L_s^2 d(L_s + \frac{1}{u}) du \\
    &\lesssim L_s^2 d h_{k}^2 (L_s +\frac{1}{t_k}),
\end{align}
thus
\begin{equation}\nonumber
    \mathbb{E}_{\vx_{t_{k+1}} \sim p_{t_{k+1}} } \left[\Vert \hat\vx_{t_k}^\vphi -\vx_{t_k} \Vert_2^2\right] \le h_k^2(L_s^2 d h_{k}^2 (L_s +\frac{1}{t_k}) + \varepsilon_{\text{sc}}^2),
\end{equation}

\begin{align}
    E_2 &\lesssim L_fL_s^{3/2}d^{1/2} \sum_{k=1}^N h_k^2 +   L_fL_sd^{1/2} \sum_{k=1}^N \frac{h_k^2}{t_k^{1/2}} + \varepsilon_{\text{sc}} \sum_{k=1}^N h_k \nonumber\\
    &\lesssim L_fL_s^{3/2}d^{1/2} h (t_n - t_1) +  L_f L_s d^{1/2} h (t_n - t_1)^{1/2} + L_f\varepsilon_{\text{sc}}(t_n - t_1)\label{E2bound}
\end{align}
as we specially designed $h_k$ for $k \in [\![1, N_1]\!]$ such that $h_k \le \frac{t_{k+1}}{2}$, which implies $\sum_{k=1}^N \frac{h_k}{t_k^{1/2}}$ is a constant-factor approximation of the integral $\int_{t_1}^{t_n} \frac{1}{t^{1/2}} \dd t \lesssim \sqrt{t_n - t_1}$.

Combining  equations \ref{E1} and \ref{E2bound}, we immediately get the result.

\end{proof}

\subsection{Proof of Theorem \ref{thm-W2-error} and Corollary \ref{thm-onestep-error}} \label{Onestep W2 error}
\begin{proof}[\bf Proof of Theorem \ref{thm-W2-error}]
Take a couple of $(\mY, \mZ) \sim \gamma(\vy,\vz)$ where $\gamma \in \Gamma(\mu, p_{t_n})$ is a coupling between  $\mu$ and $p_{t_n}$, that is,
\begin{align}
    \int_{\mathbb{R}^d} \gamma(\vy,\vz)\dd\vz &= \mu(\vy),\\
    \int_{\mathbb{R}^d}\gamma(\vy,\vz)\dd\vy &= p_{t_n}(\vz).
\end{align}

then we have $\vf_\vtheta(\mY, t_n) \sim \vf_{\vtheta,t_n}\sharp\mu$, $ \vf^{\text{ex}}(\mZ,t_n) \sim p_\delta$, thus

\begin{align}
    W_2( \vf_{\vtheta,t_n}\sharp \mu, p_\delta ) &\le \left(\mathbb{E}_{\gamma} \left[\Vert \vf_\vtheta(\mY, t_n) - \vf^{\text{ex}}(\mZ,t_n) \Vert^2_2\right] \right)^{1/2}\nonumber\\
    &\le \left(\mathbb{E}_{\gamma}\left[\Vert \vf_\vtheta(\mY, t_n) - \vf_\vtheta(\mZ,t_n) \Vert^2_2\right]\right)^{1/2} +  \left(\mathbb{E}_{\gamma}\left[\Vert \vf_\vtheta(\mZ, t_n) - \vf^{\text{ex}}(\mZ,t_n) \Vert^2_2\right]\right)^{1/2} \nonumber\\
    &\le L_f \left(\mathbb{E}_{\gamma} \left[\Vert \mY - \mZ \Vert_2^2\right] \right)^{1/2}+ t_n (\varepsilon_{\text{cm}} + L_f \varepsilon_{\text{sc}} + L_f L_s^{\frac{3}{2}} d^{\frac{1}{2}} h) + t_n^{\frac{1}{2}} L_f L_s d^{\frac{1}{2}} h.\label{W2bound}
\end{align}
Note that $\gamma$ can be any coupling between $\mu$ and $p_{t_n}$, this implies
\begin{equation} \nonumber
     W_2( \vf_{\vtheta,t_n}\sharp \mu, p_\delta ) \le  L_f W_2(\mu,p_{t_n})+ t_n (\varepsilon_{\text{cm}} + L_f \varepsilon_{\text{sc}} + L_f L_s^{\frac{3}{2}} d^{\frac{1}{2}} h) + t_n^{\frac{1}{2}} L_f L_s d^{\frac{1}{2}} h
\end{equation}
\end{proof}

\begin{proof}[\bf Proof of Corollary \ref{thm-onestep-error}]
    We first proof that
    \begin{equation} \label{onestep err}
        W_2(\vf_{\vtheta,T}\sharp \mathcal{N}(\vzero,\mI_d),p_\delta) \lesssim (d^{\frac{1}{2}}\vee m_2)L_f e^{-T} + T (\varepsilon_{\text{cm}} + L_f \varepsilon_{\text{sc}} + L_f L_s^{\frac{3}{2}} d^{\frac{1}{2}} h).
    \end{equation}
    This follows directly from the fact that $e^{-T}\vx_0 + \sqrt{1-e^{-2T}}\vxi \sim p_T(\vx),$ if $\vxi \sim \mathcal{N}(\vzero,\mI_d)$, thus
    \begin{equation} \nonumber
        W_2(\mathcal{N}(\vzero,(1-e^{-2T})\mI_d), p_T ) \le \left(\mathbb{E}_{\pdata} [\Vert e^{-T}\vx_0 + (\sqrt{1-e^{-2T}}-1)\xi \Vert_2^2]\right)^{1/2} \lesssim (\sqrt{d}\vee m_2)e^{-T}
    \end{equation}
    and Theorem \ref{thm-W2-error}, where we taking $\mu = \mathcal{N}(\vzero,\mI_d)$.
    
    The corollary then follow from a simple triangular inequality and the fact that
    \begin{align}\nonumber
    W_2(p_\delta, p_0 ) &\le \left(\mathbb{E}_{\pdata} [\Vert (1-e^{-\delta})\vx_0 + (\sqrt{1-e^{-2\delta}})\xi \Vert_2^2]\right)^{1/2} \\
    &\le ((1-e^{-\delta})^2m_2^2 + (1-e^{-2\delta})d )^{1/2}\nonumber\\
    &\lesssim (\sqrt{d}\vee m_2)\sqrt{\delta}.
    \end{align}
\end{proof}

\subsection{Proofs for Corollary \ref{multisteperr} and  \ref{MultistepW2Err} } \label{Multistep proof}
\begin{proof}[\bf Proof of Corollary \ref{multisteperr}]
    Take a couple of $(\mY,\mZ) \sim \gamma(\vy,\vz)$ where $\gamma \in \Gamma(q_{k-1}, p_\delta)$, take $\vxi \sim \mathcal{N}(\vzero,\mI_d),$ then we have
    \begin{align}
        \hat\mY = e^{-(t_{k_n} - \delta)}\mY + \sqrt{1-e^{-2(t_{k_n} - \delta)}}\vxi &\sim \mu_k, \nonumber\\
        \hat\mZ = e^{-(t_{k_n} - \delta)}\mZ + \sqrt{1-e^{-2(t_{k_n} - \delta)}}\vxi &\sim p_{t_{n_k}},
    \end{align}

    The statement follows from the fact that 
    \begin{align}
        W_2( \vf_{\vtheta,t_{n_k}}\sharp \mu_k, p_{\delta} ) &\lesssim L_f W_2(\mu_k,p_{t_{n_k}})+ t_n (\varepsilon_{\text{cm}} + L_f \varepsilon_{\text{sc}} + L_f L_s^{\frac{3}{2}} d^{\frac{1}{2}} h) + t_n^{\frac{1}{2}} L_f L_s d^{\frac{1}{2}} h \nonumber\\
        &\le  L_f (\mathbb{E}_{\gamma} \Vert \hat \mY - \hat \mZ \Vert_2^2 )^{1/2} + t_n (\varepsilon_{\text{cm}} + L_f \varepsilon_{\text{sc}} + L_f L_s^{\frac{3}{2}} d^{\frac{1}{2}} h) + t_n^{\frac{1}{2}} L_f L_s d^{\frac{1}{2}} h \nonumber\\
        &=  L_fe^{-(t_{k_n} - \delta)}(\mathbb{E}_{\gamma} \Vert  \mY - \mZ \Vert_2^2)^{1/2}+ t_n (\varepsilon_{\text{cm}} + L_f \varepsilon_{\text{sc}} + L_f L_s^{\frac{3}{2}} d^{\frac{1}{2}} h) + t_n^{\frac{1}{2}} L_f L_s d^{\frac{1}{2}} h
    \end{align}
    and $\gamma$ can be arbitrary coupling between $q_{k-1}$ and $p_{\delta}$, which means 
    $$ W_2(q_k, p_\delta) \lesssim  L_fe^{-(t_{n_k}-\delta)} W_2(q_{k-1},p_\delta) +  t_{n_k} (\varepsilon_{\text{cm}} + L_f \varepsilon_{\text{sc}} + L_f L_s^{\frac{3}{2}} d^{\frac{1}{2}} h) + t_{n_k}^{\frac{1}{2}} L_f L_s d^{\frac{1}{2}} h.$$
\end{proof}

\begin{proof} [\bf Proof of Corollary \ref{MultistepW2Err}]
For the first statement, fix $n_k \equiv \hat n $, note that $L_s \ge 1, t_{\hat n} \ge \log (2L_f) + \delta$, according to the proof of Corollary \ref{multisteperr}, we may assume $C$ is the constant factor such that 
\begin{equation} \nonumber
    W_2(q_k, p_\delta) \le L_f e^{-(t_{\hat n} - \delta)} W_2(q_{k-1},p_\delta) +  C t_{\hat n} (\varepsilon_{\text{cm}} + L_f \varepsilon_{\text{sc}} + L_f L_s^{\frac{3}{2}} d^{\frac{1}{2}} h).
\end{equation}
where we omit the term with $t_{\hat n}^{\frac{1}{2}}$ as it is controlled by the terms with $t_{\hat n}$.
Denote 
\begin{equation} \nonumber
    D = C(\varepsilon_{\text{cm}} + L_f \varepsilon_{\text{sc}} + L_f L_s^{\frac{3}{2}} d^{\frac{1}{2}} h)\text{ and } \mathcal{E}_k = W_2(q_k,\pdata)
\end{equation}
for short, we have
\begin{equation} \nonumber
    \mathcal{E}_k \le L_f e^{-(t_{\hat n}-\delta)} \mathcal{E}_{k-1} + t_{\hat n}D,
\end{equation}
which means 
\begin{equation} \nonumber
    \mathcal{E}_k-  \frac{t_{\hat n}D}{1-L_f e^{-(t_{\hat n}-\delta)}} \le L_f e^{-(t_{\hat n}-\delta)} \left(\mathcal{E}_{k-1}-   \frac{t_{\hat n}D}{1-L_f e^{-(t_{\hat n}-\delta)}} \right),
\end{equation}
thus,
\begin{enumerate}
    \item if $\mathcal{E}_{k-1}\le  \frac{t_{\hat n}D}{1-L_f e^{-(t_{\hat n}-\delta)}}$,  $\mathcal{E}_K\le  \frac{t_{\hat n}D}{1-L_f e^{-(t_{\hat n}-\delta)}}$ for all $K \ge k$;
    \item if $\mathcal{E}_{k-1} >  \frac{t_{\hat n}D}{1-L_f e^{-(t_{\hat n}-\delta)}}$,
    \begin{equation}
        \mathcal{E}_k  \le    \frac{t_{\hat n}D}{1-L_f e^{-(t_{\hat n}-\delta)}} + \left(L_f e^{-(t_{\hat n}-\delta)}\right)^{k-1}\left(\mathcal{E}_{1}-   \frac{t_{\hat n}D}{1-L_f e^{-(t_{\hat n}-\delta)}} \right).
    \end{equation}
\end{enumerate}
These observation shows that $\mathcal{E}_k$ is exponentially upper-bounded by $ \frac{t_{\hat n}D}{1-L_f e^{-(t_{\hat n}-\delta)}}$, and such an upper bound can be further minimized over $\hat n \in [\![ 1,N ]\!]$: in fact if we take $t_{\hat n} \approx \log(2L_f)+\delta$, then $$\frac{t_{\hat n}D}{1-L_f e^{-(t_{\hat n}-\delta)}} \approx 2\log(2L_f)D = O(\log (L_f) D).$$
together with $L_f e^{-(t_{\hat n} - \delta)} \approx \frac{1}{2}$, $\mathcal{E}_1 = O((d^{\frac{1}{2}}\vee m_2)L_fe^{-T} +T(\varepsilon_{\text{cm}} + L_f \varepsilon_{\text{sc}} + L_f L_s^{\frac{3}{2}} d^{\frac{1}{2}} h))$,
this completes the proof of the first statement.

The second statement follows from the fact that 
\begin{align}\nonumber
    W_2(p_\delta, p_0 ) &\le \left(\mathbb{E}_{\pdata} [\Vert (1-e^{-\delta})\vx_0 + (\sqrt{1-e^{-2\delta}})\xi \Vert_2^2]\right)^{1/2} \\
    &\le ((1-e^{-\delta})^2m_2^2 + (1-e^{-2\delta})d )^{1/2}\nonumber\\
    &\lesssim (\sqrt{d}\vee m_2)\sqrt{\delta}.
    \end{align}
and a simple use of triangular inequality: $W_2(q_k,p_0) \le W_2(q_k,p_\delta) + W_2(p_\delta, p_0)$.
\end{proof}

\subsection{Proof of Corollary \ref{end-to-end TV}} \label{OUproof}
Before we prove the corollary \ref{end-to-end TV}, we first prove the following lemma, which shows that TV error can be bounded after a small time OU regularization.
\begin{lem}\label{OU regularization}
    For any two distribution $p$ and $q$, running the OU process \ref{forwardOU} for $p, q$ individually with time $\tau>0$, the following TV distance bound holds,
    \begin{equation} \nonumber
        \text{TV}(p P_{\text{OU}}^\tau, q P_{\text{OU}}^\tau) \lesssim \frac{1}{\sqrt{\tau}} W_1(p,q) \le \frac{1}{\sqrt{\tau}} W_2(p,q) 
    \end{equation}
\end{lem}

\begin{proof}
    Denote $\psi_{\sigma^2}(\vy)$ as the density function to the normal distribution $\mathcal{N}(\vzero, \sigma^2 \mI_d)$. We write the $\text{TV}(p P_{\text{OU}}^\tau, q P_{\text{OU}}^\tau)$ into integral form as:
    \begin{align}
        \text{TV}(p P_{\text{OU}}^\tau, q P_{\text{OU}}^\tau) &= \frac{1}{2}\int_{\mathbb{R}^d} \vert (p P_{\text{OU}}^\tau)(\vx) - (q P_{\text{OU}}^\tau)(\vx)\vert \dd \vx \nonumber\\
        &=  \frac{1}{2}\int_{\mathbb{R}^d} \left\vert \int_{\mathbb{R}^d} p(\vy) \psi_{1-e^{-2\tau}} (\vx - e^{-\tau}\vy) \dd \vy - \int_{\mathbb{R}^d} q(\vz) \psi_{1-e^{-2\tau}} (\vx - e^{-\tau}\vz) \dd \vz   \right\vert \dd \vx. 
    \end{align}
    Take a coupling $\gamma \in \Gamma(p,q)$, then 
    \begin{align}
        \int_{\mathbb{R}^d} \gamma(\vy, \vz) \dd \vz =  p(\vy),\nonumber\\
        \int_{\mathbb{R}^d} \gamma(\vy, \vz) \dd \vy =  q(\vz),
    \end{align}
    we have
    \begin{align}
        \text{TV}(p P_{\text{OU}}^\tau, q P_{\text{OU}}^\tau) &= \frac{1}{2} \int_{\mathbb{R}^d} \left\vert \int_{\mathbb{R}^{\dd \times d}} \gamma(\vy,\vz) [\psi_{1-e^{-2\tau}} (\vx - e^{-\tau}\vy) - \psi_{1-e^{-2\tau}} (\vx - e^{-\tau}\vz)] \dd \vy \dd \vz    \right\vert \dd \vx.  \nonumber\\
        &\le \frac{1}{2}\int_{\mathbb{R}^d}\int_{\mathbb{R}^{d\times d}} \gamma(\vy,\vz) \left\vert \psi_{1-e^{-2\tau}} (\vx - e^{-\tau}\vy) - \psi_{1-e^{-2\tau}} (\vx - e^{-\tau}\vz) \right\vert \dd \vy \dd \vz \dd \vx \nonumber\\
        &= \int_{\mathbb{R}^{d\times d}}\gamma(\vy,\vz)  \left( \frac{1}{2}\int_{\mathbb{R}^d} \left\vert \psi_{1-e^{-2\tau}} (\vx - e^{-\tau}\vy) - \psi_{1-e^{-2\tau}} (\vx - e^{-\tau}\vz) \right\vert \dd \vx \right) \dd \vy \dd \vz \nonumber\\
        &= \int_{\mathbb{R}^{d\times d}}\gamma(\vy,\vz)  \text{TV}(\psi_{1-e^{-2\tau}}(\cdot - e^{-\tau}\vy),\psi_{1-e^{-2\tau}}(\cdot - e^{-\tau}\vz) )  \dd \vy \dd \vz \nonumber\\
        &\le \int_{\mathbb{R}^{d\times d}}\gamma(\vy,\vz)  \sqrt{\frac{1}{2}\text{KL}(\psi_{1-e^{-2\tau}}(\cdot - e^{-\tau}\vy)\Vert\psi_{1-e^{-2\tau}}(\cdot - e^{-\tau}\vz) ) } \dd \vy \dd \vz \nonumber\\
        &=  \int_{\mathbb{R}^{d\times d}}\gamma(\vy,\vz)  \frac{1}{2}\sqrt{\frac{e^{-2\tau}}{1-e^{-2\tau}}  \Vert \vy - \vz \Vert_2^2 } \dd\vy\dd\vz \nonumber\\
        &=  \frac{1}{2\sqrt{e^{2\tau}-1}}\int_{\mathbb{R}^{d\times d}}\gamma(\vy,\vz) \Vert \vy - \vz \Vert_2 \dd\vy\dd\vz
    \end{align}
    Noting $\frac{1}{2\sqrt{e^{2\tau}-1}} \le \frac{1}{2\sqrt{2 \tau}}$, and taking $\gamma$ over all coupling $\Gamma(p,q)$, we have 
    \begin{align} \nonumber
        \text{TV}(p P_{\text{OU}}^\tau, q P_{\text{OU}}^\tau) \lesssim \frac{1}{\sqrt{\tau}} W_1(p,q) \le \frac{1}{\sqrt{\tau}} W_2(p,q)
    \end{align}
\end{proof}

\begin{proof}[\bf Proof of Corollary \ref{end-to-end TV}]
    According to the triangular inequality, Lemma \ref{OU regularization} and \eqref{onestep err to zero},
    \begin{align}
        \text{TV}(q_1 P_{\text{OU}}^\delta, \pdata) &\le \text{TV}(q_1 P_{\text{OU}}^\delta, p_\delta P_{\text{OU}}^\delta) + \text{TV}(p_\delta P_{\text{OU}}^\delta, \pdata) \nonumber\\
        &\lesssim \frac{1}{\sqrt{\delta}} W_2(q_1,p_\delta) + \text{TV}(p_{2\delta}, \pdata) \nonumber\\
        &\lesssim\frac{1}{\sqrt{\delta}} \left((d^{\frac{1}{2}} \vee m_2) L_f e^{-T} + T (\varepsilon_{\text{cm}} + L_f \varepsilon_{\text{sc}} + L_f L_s^{\frac{3}{2}} d^{\frac{1}{2}} h)\right) + \text{TV}(p_{2\delta}, \pdata).
    \end{align}
    Note that if we take $\delta \asymp \frac{\varepsilon^2}{L_s^2(d\vee m_2^2)}$, then by Lemma 6.4, \cite{Lee2022ConvergenceOS}, $\text{TV}(p_{2\delta}, \pdata) \le \varepsilon$, this concludes that 
    \begin{equation} \nonumber
        \text{TV}(q_1 P_{\text{OU}}^\delta, \pdata) \lesssim \frac{L_s  (d^{\frac{1}{2}} \vee m_2)}{\varepsilon} [(\log L_f + \frac{T}{2^k})   (\varepsilon_{\text{cm}} + L_f \varepsilon_{\text{sc}} + L_f L_s^{\frac{3}{2}} d^{\frac{1}{2}} h)+ \frac{(d^{\frac{1}{2}}\vee m_2)L_f}{2^ke^T}] + \varepsilon. 
    \end{equation}
    Similarly for $q_k$, if we take $\delta \asymp \frac{\varepsilon^2}{L_s^2(d\vee m_2^2)}$
\begin{align}
        &\text{TV}(q_k P_{\text{OU}}^\delta, \pdata) \\
        \le &\text{TV}(q_k P_{\text{OU}}^\delta, p_\delta P_{\text{OU}}^\delta) + \text{TV}(p_\delta P_{\text{OU}}^\delta, \pdata) \nonumber\\
        \lesssim &\frac{L_s  (d^{\frac{1}{2}} \vee m_2)}{\varepsilon} [(\log L_f + \frac{T}{2^k})   (\varepsilon_{\text{cm}} + L_f \varepsilon_{\text{sc}} + L_f L_s^{\frac{3}{2}} d^{\frac{1}{2}} h)+ \frac{(d^{\frac{1}{2}}\vee m_2)L_f}{2^ke^T}] + \varepsilon. 
\end{align}
\end{proof}

\subsection{Proof of Corollary \ref{TV-Langevin-corrector}} \label{ULMCproof}
We first introduce the following lemma which originally comes from Theorem 5 in \cite{Chen2023ThePF}.
\begin{lem}[Underdamped corrector, Theorem 5 in \cite{Chen2023ThePF}]\label{ULMClemma}
    Denote the Markov kernel $\hat P_{\text{ULMC}}$ to be defined by the equation \ref{ULMC}. Suppose $\gamma \asymp L_s$, $\nabla U$ is $L_s$-Lipschitz, $p \propto \exp(-U)$, and $\mathbb{E}_{\vx \sim q} \Vert \vs(\vx) - (-\nabla U(\vx)) \Vert_2^2 \le \varepsilon_{\text{sc}}^2$. Denote $\vp := p \otimes \mathcal{N}(\vzero,\mI_d)$, and $\vq := q \otimes \mathcal{N}(\vzero,\mI_d)$. For any $T_{\text{corr}} := N\tau \lesssim 1/{\sqrt{L_s}}$,
    \begin{equation} \nonumber
        \text{TV}(\vq\hat P_{\text{ULMC}}^N,\vp) \lesssim \frac{W_2(q,p)}{L_s^{1/4}T_{\text{corr}}^{3/2}} + \frac{\varepsilon_{\text{sc}} T^{1/2}_{\text{corr}}}{L_s^{1/4}} + L_s^{3/4}T_{\text{corr}}^{1/2}d^{1/2}\tau.
    \end{equation}
    In particular, for $T_{\text{corr}} \asymp 1/\sqrt{L_s},$
    \begin{equation} \nonumber
        \text{TV}(\vq\hat P_{\text{ULMC}}^N,\vp)  \lesssim \sqrt{L_s}W_2(q,p) + \varepsilon_{\text{sc}} /\sqrt{L_s} + \sqrt{L_sd} \tau.
    \end{equation}
\end{lem}
\begin{proof}[\bf Proof of Corollary \ref{TV-Langevin-corrector}]
    Given any distribution $q$ on $\mathbb{R}^d$, we write $q\hat P_{\text{ULMC}}$ to denote the projection onto the $\vz-$coordinates of $\vq \hat P_{\text{ULMC}}$. Now according to \ref{ULMClemma}, we have
    \begin{align}
        \text{TV}(q_1\hat P_{\text{ULMC}}^N, \pdata) &\le \text{TV}(q_1\hat P_{\text{ULMC}}^N, p_\delta) + \text{TV}(p_\delta, \pdata) \nonumber\\
        &\lesssim \sqrt{L_s} W_2(q_1,p_\delta) + \varepsilon_{\text{sc}} /\sqrt{L_s} + \sqrt{L_sd} \tau + \text{TV}(p_\delta, \pdata) .
    \end{align}
    According to \eqref{onestep err to zero}, we have
    \begin{align}
        \text{TV}&(q_1\hat P_{\text{ULMC}}^N,\pdata) \nonumber\\
        &\lesssim (d^{\frac{1}{2}}\vee m_2)L_fL_s^{\frac{1}{2}} e^{-T} + TL_s^{\frac{1}{2}} ( \varepsilon_{\text{cm}} + L_f\varepsilon_{\text{sc}} + L_f L_s^{\frac{3}{2}} d^{\frac{1}{2}} h ) + L_s^{-\frac{1}{2}} \varepsilon_{\text{sc}} + L_s^{\frac{1}{2}}d^{\frac{1}{2}}\tau + \varepsilon.
    \end{align}
    Similarly, we have 
    \begin{align}
        \text{TV}&(q_k\hat P_{\text{ULMC}}^N,\pdata) \nonumber\\
        &\lesssim (\log( L_f) + \frac{T}{2^k}) L_s^{\frac{1}{2}}(\varepsilon_{\text{cm}} + L_f\varepsilon_{\text{sc}} + L_f L_s^{\frac{3}{2}}d^{\frac{1}{2}}h) + \frac{(d^{\frac{1}{2}}\vee m_2)L_s^{\frac{1}{2}}L_f}{2^{k}e^T} + L_s^{-\frac{1}{2}} \varepsilon_{\text{sc}} + L_s^{\frac{1}{2}}d^{\frac{1}{2}}\tau+ \varepsilon.
    \end{align}
\end{proof}

\subsection{Proof of Lemma \ref{Hessian Bound Lemma} and Corollary \ref{Bounded pdata err}} \label{bound support proof}

\begin{proof}[\bf Proof of Lemma \ref{Hessian Bound Lemma}]
    Let $\mu_{\vx, \sigma^2}$ denote the density $\mu( \dd \vu)$ weighted with the gaussian $\psi_{\sigma^2}(\vu - \vx) \sim e^{-\frac{\Vert \vx - \vu\Vert_2^2}{2\sigma^2}},$ that is, 
    \begin{equation}\nonumber
        \mu_{\vx, \sigma^2}(\dd\vu) = \frac{e^{-\frac{\Vert \vx - \vu\Vert_2^2}{2\sigma^2} }\mu(\dd\vu)}{\int_{\mathbb{R}^d} e^{-\frac{\Vert \vx - \tilde\vu\Vert_2^2}{2\sigma^2}} \mu(\dd\tilde\vu)}.
    \end{equation}
    Note that
    \begin{equation}\nonumber
        \nabla \log (\mu \ast \psi_{\sigma^2}(\vx)) = \frac{\nabla\int_{ \mathbb{R}^d} e^{-\frac{\Vert \vx - \vu\Vert_2^2}{2\sigma^2} }\mu(\dd\vu)}{\int_{\mathbb{R}^d} e^{-\frac{\Vert \vx - \vu\Vert_2^2}{2\sigma^2}} \mu(\dd\vu)} = \frac{\int_{ \mathbb{R}^d} -\frac{\vx - \vu}{\sigma^2}e^{-\frac{\Vert \vx - \vu\Vert_2^2}{2\sigma^2} }\mu(\dd\vu)}{\int_{\mathbb{R}^d} e^{-\frac{\Vert \vx - \vu\Vert_2^2}{2\sigma^2}} \mu(\dd\vu)} = -\frac{1}{\sigma^2} \mathbb{E}_{\mu_{\vx,\sigma^2}}[\vx - \vu],
    \end{equation}
    
    \begin{align}
        \nabla^2 \log (\mu \ast \psi_{\sigma^2}(\vx)) &=   \frac{\nabla \otimes \int_{ \mathbb{R}^d} -\frac{\vx - \vu}{\sigma^2}e^{-\frac{\Vert \vx - \vu\Vert_2^2}{2\sigma^2} }\mu(\dd\vu)}{\int_{\mathbb{R}^d} e^{-\frac{\Vert \vx - \vu\Vert_2^2}{2\sigma^2}} \mu(\dd\vu)} - \left(\frac{\int_{ \mathbb{R}^d} -\frac{\vx - \vu}{\sigma^2}e^{-\frac{\Vert \vx - \vu\Vert_2^2}{2\sigma^2} }\mu(\dd\vu)}{\int_{\mathbb{R}^d} e^{-\frac{\Vert \vx - \vu\Vert_2^2}{2\sigma^2}} \mu(\dd\vu)}\right)^{\otimes2} \nonumber\\
        &= -\frac{1}{\sigma^2}\mI_d +  \frac{\int_{ \mathbb{R}^d} \left(\frac{\vx - \vu}{\sigma^2}\right)^{\otimes 2}e^{-\frac{\Vert \vx - \vu\Vert_2^2}{2\sigma^2} }\mu(\dd\vu)}{\int_{\mathbb{R}^d} e^{-\frac{\Vert \vx - \vu\Vert_2^2}{2\sigma^2}} \mu(\dd\vu)} - \left(\frac{\int_{ \mathbb{R}^d} -\frac{\vx - \vu}{\sigma^2}e^{-\frac{\Vert \vx - \vu\Vert_2^2}{2\sigma^2} }\mu(\dd\vu)}{\int_{\mathbb{R}^d} e^{-\frac{\Vert \vx - \vu\Vert_2^2}{2\sigma^2}} \mu(\dd\vu)}\right)^{\otimes 2} \nonumber\\
        &= \frac{1}{\sigma^4} \int_{\mathbb{R}^d} (\vu -\mathbb{E}_{\mu_{\vx,\sigma^2}}[\vu] ) \otimes (\vu -\mathbb{E}_{\mu_{\vx,\sigma^2}}[\vu]) \mu_{\vx, \sigma^2}(\dd\vu) - \frac{1}{\sigma^2} \mI_d
    \end{align}
    where for any vector $\vy \in \mathbb{R}^d$, we denote $\vy^{\otimes 2} = \vy \otimes \vy = \vy \vy^T \in \mathbb{R}^{d\times d}$ as a matrix and denote $\nabla \otimes \vf(\vx)$ as the Jacobbian matrix $[\frac{\partial f_i(\vx)}{\partial x_j}]_{1\le i,j \le d}$
    
    Note that if $\mu$ is bounded on a set of radius $R$, so as $\mu_{\vx, \sigma^2}$, then the covariance of $\mu_{\vx, \sigma^2}$ is bounded by $R^2$ in operator norm.
    
    Now we take $\mu(\vu) = e^{dt}p_0(e^t \vu)$, which is bounded on a set of radius $e^{-t}R$. Take $\sigma^2 = 1-e^{-2t}$, we have $\mu \ast \psi_{\sigma^2}(\vx) = p_t(\vx)$, thus
    
    $$\Vert\nabla^2 \log p_t(\vx) \Vert_{\text{op}} \le \frac{e^{-2t}R^2}{(1-e^{-2t})^2}  + \frac{1}{1-e^{-2t}}.$$
\end{proof}

\begin{proof}[\bf Proof of Corollary \ref{Bounded pdata err}]
    Note that $\text{supp } \pdata \subseteq B(\vzero,R)$ implies $\mathbb{E}_{\vx_0 \sim \pdata} [\Vert \vx_0 \Vert_2^2] \lesssim R^2$. Taking $\delta \asymp \frac{\varepsilon^2}{R^2\vee d}$, according to Lemma \ref{Hessian Bound Lemma}, we have
    $$\Vert\nabla^2 \log p_t(\vx) \Vert_{\text{op}} \le \frac{e^{-2t}R^2}{(1-e^{-2t})^2}  + \frac{1}{1-e^{-2t}} \lesssim \frac{1}{t}\vee \frac{R^2}{t^2} \lesssim \frac{R^2(R^2\vee d)^2}{\varepsilon^4}, \forall t \ge \delta $$
    and thus $\nabla \log p_t(\vx)$ satisfies Assumption \ref{Lips cond score} with $L_s  \asymp \frac{R^2(R^2\vee d)^2}{\varepsilon^4}$ for $t \ge \delta$. Now according to Corollary \ref{thm-onestep-error}, we have that 
    \begin{equation}\nonumber
        W_2(\vf_{\vtheta,T}\sharp \mathcal{N}(\vzero,\mI_d),\pdata) \lesssim (d^{\frac{1}{2}} \vee R) L_f e^{-T} + T (\varepsilon_{\text{cm}} + L_f \varepsilon_{\text{sc}} + \frac{L_f d^{\frac{1}{2}} R^3(R^2\vee d)^3h}{\varepsilon^6})  +\varepsilon,
    \end{equation}
    thus if we take $h =O(\frac{\varepsilon^7}{d^{1/2}R^3(R^6\vee d^3)L_f T})$,  $\varepsilon_{\text{cm}} = O(\frac{\varepsilon}{T}), \varepsilon_{\text{sc}} = O(\frac{\varepsilon}{L_f T})$, $T = O(\log (\frac{L_f (\sqrt{d} \vee R)}{\varepsilon}))$, we can guarantee $W_2(\vf_{\vtheta,T}\sharp \mathcal{N}(\vzero,\mI_d),\pdata) \lesssim \varepsilon$.
    
    Similarly, according to Corollary \ref{MultistepW2Err}, we have
    $$ W_2(q_k, \pdata )\lesssim  (\log( L_f) + 2^{-k}T) (\varepsilon_{\text{cm}} + L_f\varepsilon_{\text{sc}} + \frac{L_f d^{\frac{1}{2}} R^3(R^2\vee d)^3h}{\varepsilon^6}) + 2^{-k}(d^{\frac{1}{2}}\vee m_2)L_fe^{-T}+  \varepsilon,$$
    thus if we take $h =O(\frac{\varepsilon^7}{d^{1/2}R^3(R^6\vee d^3)L_f \log(L_f)})$,  $\varepsilon_{\text{cm}} = O(\frac{\varepsilon}{\log(L_f)}), \varepsilon_{\text{sc}} = O(\frac{\varepsilon}{L_f \log(L_f)})$, $k = O(\log(T \vee (\frac{(d\vee m_2^2)L_fL_s}{\varepsilon})))$, we can guarantee $W_2(q_k, \pdata )\lesssim \varepsilon.$
\end{proof}

\section{Additional proofs} \label{additional proof}

\begin{lem}\label{creatingscore} Assuming \ref{score err}  and \ref{consistency err}, then
$\hat\vs_\vtheta(\vx) := \frac{\vf_\vtheta(\vx, t_2) - e^{h_1}\vx}{e^{h_1} - 1}$ is a score model approximating $\nabla \log p_{t_2}(\vx)$ with $L^2$ error $\sqrt{\varepsilon^2_{\text{cm}} + \varepsilon^2_{\text{sc}}}$.
\end{lem}
\begin{proof}
    Letting $h_1 = t_2 - t_1 = t_2 - \delta$, the consistency loss assumption \ref{consistency err} ensures 
$$\mathbb{E}_{\vx_{t_2}\sim p_{t_2}} \left[\Vert \vf_{\vtheta}(\vx_{t_2}, t_2) - \vf_\vtheta(\hat\vx^\phi_{t_1},t_1)\Vert_2^2\right] \le \varepsilon_{\text{cm}}^2 h_1^2,$$
where we have $ \vf_\vtheta(\hat\vx^\phi_{t_1},t_1) =\hat\vx^\phi_{t_1} = e^{h_1}\vx_{t_2} + (e^{h_1} - 1) \vs_\vphi(\vx_{t_2},t_2) $, thus
$$\mathbb{E}_{\vx_{t_2}\sim p_{t_2}} \left[ \left\Vert \frac{\vf_{\vtheta}(\vx_{t_2}, t_2) - e^{h_1}\vx_{t_2}}{e^{h_1}-1} - \vs_\vphi(\vx_{t_2},t_2)\right\Vert_2^2 \right] \le \varepsilon_{\text{cm}}^2 \frac{h_1^2}{(e^{h_1} -1)^2} \le \varepsilon_{\text{cm}}^2.$$
\end{proof}

\begin{lem}\label{CTform} Under the OU scheduler \ref{OUscheduler}, let $\Delta t:= \max_{1\le n\le N-1} (t_{n+1} - t_n)$. Assume $\vf_{\vtheta^-}$ is twice continuously differentiable with bounded second derivatives, and $\mathbb{E}[\Vert \nabla \log p_{t_{n}}(\vx_{t_n})\Vert_2^2] \le \infty$. Assume the CD objective \ref{distillation_loss} is defined with the exponential integrator \ref{ODEsolver} and exact score model $\vs_\vphi(\vx,t) = \nabla \log p_t(\vx)$, then we can define the CT objective as
\begin{equation} \label{CTobject} 
    \mathcal{L}^N_{\text{CT}}(\vtheta,\vtheta^-):= \mathbb{E}[\Vert \vf_\vtheta(e^{-t_{n+1}}\vx_0 + \sqrt{1-e^{-2t_{n+1}}} \vz, t_{n+1}) - \vf_{\vtheta^-}(e^{-t_{n}}\vx_0 + \frac{1-e^{-(t_{n} + t_{n+1})}}{\sqrt{1-e^{-2t_{n+1}}}}\vz, t_n)\Vert_2^2],
\end{equation}
where $\vx_0 \sim \pdata$, $\vz \sim\mathcal{N}(\vzero,\mI_d)$. Then we have
\begin{equation} \nonumber
\frac{\partial}{\partial \vtheta}\mathcal{L}^N_{\text{CT}}(\vtheta,\vtheta^-) = \frac{\partial}{\partial \vtheta}\mathcal{L}^N_{\text{CD}}(\vtheta,\vtheta^-) + O((\Delta t)^2)
\end{equation}
\end{lem}
\begin{proof}
    We first prove that, if $\vx_0 \sim \pdata, \vz \sim \mathcal{N}(\vzero,\mI_d), \vx_t = e^{-t}\vx_0 + \sqrt{1-e^{-2t}}\vz$, then we have $\nabla\log p_t(\vx) = - \mathbb{E}[\frac{\vz}{\sqrt{1-e^{-2t}}} \vert \vx_t]$. This comes from the fact that
    \begin{align*}
        \nabla \log p_t(\vx_t) &= \frac{\int_{\mathbb{R}^d} \pdata(\vx_0) \nabla_{\vx_t} p(\vx_t \vert \vx_0)\dd \vx_0}{\int_{\mathbb{R}^d}\pdata(\vx_0)p(\vx_t\vert \vx_0) \dd \vx_0} \\
        &= \frac{\int_{\mathbb{R}^d} \pdata(\vx_0) p(\vx_t\vert\vx_0) \nabla_{\vx_t}\log p(\vx_t \vert \vx_0)\dd \vx_0}{p_t(\vx_t)} \\
        &= \int_{\mathbb{R}^d} \frac{\pdata(\vx_0) p(\vx_t\vert\vx_0)}{p_t(\vx_t)} \nabla_{\vx_t} \log p(\vx_t\vert \vx_0) \dd \vx_0 \\
        &= \int_{\mathbb{R}^d} p(\vx_0\vert\vx_t) \nabla_{\vx_t} \log p(\vx_t\vert\vx_0)\dd\vx_0 \\
        &= \mathbb{E}[\nabla_{\vx_t} \log p(\vx_t \vert \vx_0) \vert \vx_t],
    \end{align*}
    and  $p(\vx_t \vert \vx_0)  \sim e^{- \frac{\Vert \vx_t - e^{-t}\vx_0\Vert^2}{2(1-e^{-2t})}}$, which means
    \begin{equation*}
        \nabla_{\vx_t} \log p(\vx_t \vert \vx_0) = -\frac{\vx_t - e^{-t}\vx_0}{1-e^{-2t}} = -\frac{\vz}{\sqrt{1-e^{-2t}}}.
    \end{equation*}
    Now we may rewrite $\frac{\partial}{\partial \vtheta}\mathcal{L}^N_{\text{CD}}(\vtheta,\vtheta^-)$ as
    \begin{align*}
        &\frac{\partial}{\partial \vtheta}\mathcal{L}^N_{\text{CD}}(\vtheta,\vtheta^-) = \frac{\partial}{\partial \vtheta}\mathbb{E}[\Vert \vf_\vtheta(\vx_{t_{n+1}}, t_{n+1})- \vf_{\vtheta^-} \left(e^{t_{n+1} - t_n}\vx_{t_{n+1}} + (e^{t_{n+1} - t_n} - 1) \nabla \log p_{t_{n+1}}(\vx_{t_{n+1}}), t_n\right) \Vert_2^2] \\
        =&  \frac{\partial}{\partial \vtheta}\mathbb{E}\left[\Vert \vf_\vtheta(\vx_{t_{n+1}}, t_{n+1})   \Vert_2^2 - 2\left\langle  \vf_\vtheta(\vx_{t_{n+1}}, t_{n+1}),  \vf_{\vtheta^-} \left(e^{t_{n+1} - t_n}\vx_{t_{n+1}} + (e^{t_{n+1} - t_n} - 1) \nabla \log p_{t_{n+1}}(\vx_{t_{n+1}}), t_n\right)  \right\rangle \right],
    \end{align*}
    and similarly, 
    \begin{align*}
        &\frac{\partial}{\partial \vtheta}\mathcal{L}^N_{\text{CT}}(\vtheta,\vtheta^-) = \frac{\partial}{\partial \vtheta}\mathbb{E}[\Vert \vf_\vtheta(\vx_{t_{n+1}}, t_{n+1})- \vf_{\vtheta^-} \left(e^{-t_{n}}\vx_0 + \frac{1-e^{-(t_{n} + t_{n+1})}}{\sqrt{1-e^{-2t_{n+1}}}}\vz, t_n\right) \Vert_2^2] \\
        =&  \frac{\partial}{\partial \vtheta}\mathbb{E}\left[\Vert \vf_\vtheta(\vx_{t_{n+1}}, t_{n+1})   \Vert_2^2 - 2\left\langle  \vf_\vtheta(\vx_{t_{n+1}}, t_{n+1}),  \vf_{\vtheta^-} \left(e^{-t_{n}}\vx_0 + \frac{1-e^{-(t_{n} + t_{n+1})}}{\sqrt{1-e^{-2t_{n+1}}}}\vz, t_n\right)  \right\rangle \right],
    \end{align*}
    thus
    \begin{align*}
        &\frac{\partial}{\partial \vtheta}\mathcal{L}^N_{\text{CD}}(\vtheta,\vtheta^-) - \frac{\partial}{\partial \vtheta}\mathcal{L}^N_{\text{CT}}(\vtheta,\vtheta^-)  \\
        = &2\mathbb{E}[\langle \frac{\partial}{\partial \vtheta} \vf_\vtheta(\vx_{t_{n+1}}, t_{n+1}),\\
        &\vf_{\vtheta^-} \left(e^{-t_{n}}\vx_0 + \frac{1-e^{-(t_{n} + t_{n+1})}}{\sqrt{1-e^{-2t_{n+1}}}}\vz,  t_n\right) - \vf_{\vtheta^-} (e^{t_{n+1} - t_n}\vx_{t_{n+1}} + (e^{t_{n+1} - t_n} - 1) \nabla \log p_{t_{n+1}}(\vx_{t_{n+1}}), t_n)   \rangle].
    \end{align*}
    Notice that
    \begin{align*}
        &\vf_{\vtheta^-} \left(e^{-t_{n}}\vx_0 + \frac{1-e^{-(t_{n} + t_{n+1})}}{\sqrt{1-e^{-2t_{n+1}}}}\vz,  t_n\right) - \vf_{\vtheta^-} (e^{t_{n+1} - t_n}\vx_{t_{n+1}} + (e^{t_{n+1} - t_n} - 1) \nabla \log p_{t_{n+1}}(\vx_{t_{n+1}}), t_n) \\
        =& \nabla_{\vx} \vf_{\vtheta^-} (e^{t_{n+1} - t_n}\vx_{t_{n+1}} + (e^{t_{n+1} - t_n} - 1) \nabla \log p_{t_{n+1}}(\vx_{t_{n+1}}), t_n)  \cdot \\
        &\left( \frac{1-e^{-(t_{n} + t_{n+1})}}{\sqrt{1-e^{-2t_{n+1}}}}\vz  - e^{t_{n+1} - t_n} \sqrt{1-e^{-2t_n}}\vz + (e^{t_{n+1} - t_n} - 1)\nabla \log p_{t_{n+1}}(\vx_{t_{n+1}})\right) + O((t_{n+1} - t_n)^2)\\
        =&\nabla_{\vx} \vf_{\vtheta^-} (e^{t_{n+1} - t_n}\vx_{t_{n+1}} + (e^{t_{n+1} - t_n} - 1) \nabla \log p_{t_{n+1}}(\vx_{t_{n+1}}), t_n)  \cdot \\
        &(e^{t_{n+1} - t_n} - 1) \left(\frac{\vz}{\sqrt{1-e^{-2t_{n+1}}}}  - \mathbb{E}\left[\left.\frac{\vz}{\sqrt{1-e^{-2t_{n+1}}}} \right\vert \vx_{t_{n+1}}\right]\right)+ O((t_{n+1} - t_n)^2),
    \end{align*}
    we have (writing the terms of $\vx_{t_{n+1}}, t_{n+1}$ with $C(\vx_{t_{n+1}}, t_{n+1})$ for simplicity)
    \begin{align*}
        &\frac{\partial}{\partial \vtheta}\mathcal{L}^N_{\text{CD}}(\vtheta,\vtheta^-) - \frac{\partial}{\partial \vtheta}\mathcal{L}^N_{\text{CT}}(\vtheta,\vtheta^-) \\
        =&\mathbb{E}[ C(\vx_{t_{n+1}}, t_{n+1}) (\vz - \mathbb{E}[\vz \vert \vx_{t_{n+1}}]) ] + O((t_{n+1} - t_n)^2) \\
        =& \mathbb{E}[ C(\vx_{t_{n+1}}, t_{n+1}) \vz] - \mathbb{E}[ \mathbb{E}[  C(\vx_{t_{n+1}}, t_{n+1}) \vz \vert \vx_{t_{n+1}}]]+ O((t_{n+1} - t_n)^2) \\
        =&O((t_{n+1} - t_n)^2)
    \end{align*}
    which finished the proof.
\end{proof}

% \begin{align} \nonumber
%     \vf_\vtheta(\vx_t,t) &\approx \vf_\vtheta(\vx_s,\tau),~\forall~ 0 \le \tau,t \le T,\\
%     \vf_\vtheta(\vx, \delta) &= \vx
% \end{align}

\end{document}